\documentclass{amsart}

\usepackage{amsthm,amsfonts,amssymb}
\usepackage[usenames]{color}

\newtheorem{theorem}{Theorem}[section]
\newtheorem{lemma}[theorem]{Lemma}

\newtheorem{proposition}[theorem]{Proposition}
\theoremstyle{remark}
\newtheorem{remark}[theorem]{Remark}
\newtheorem{definition}[theorem]{Definition}

\numberwithin{equation}{section}

\newcommand{\C}{\mathbb{C}}

\newcommand{\N}{\mathbb{N}}
\newcommand{\R}{\mathbb{R}}
\newcommand{\Z}{\mathbb{Z}}
\begin{document}

\title[Fourier multipliers and cylindrical boundary value problems]{Discrete Fourier multipliers and cylindrical boundary value problems}

\author{R.\ Denk, T.\ Nau}
\address{University of Konstanz, Department of Mathematics and Statistics, 78457 Konstanz, Germany}
\email{robert.denk@uni-konstanz.de, tobias.nau@uni-konstanz.de}


\begin{abstract}
We consider operator-valued boundary value problems in $(0,2\pi)^n$ with periodic or, more generally, $\nu$-periodic boundary conditions. Using the concept of discrete vector-valued Fourier multipliers, we give equivalent conditions for the unique solvability of the boundary value problem. As an application, we study vector-valued parabolic initial boundary value problems in cylindrical domains $(0,2\pi)^n\times V$ with $\nu$-periodic boundary conditions in the cylindrical directions. We show that under suitable assumptions on the coefficients, we obtain maximal $L^q$-regularity for such problems.
\end{abstract}

\maketitle

%%%%%%%%%%%%%%%%%%%%%%%%%%%%%%%%%%%%%%%%%%%%%%%%%%%%%%%%%%%%%%%%%%
\section{Introduction}
%%%%%%%%%%%%%%%%%%%%%%%%%%%%%%%%%%%%%%%%%%%%%%%%%%%%%%%%%%%%%%%%%%

In this paper we  first study boundary value problems with operator-valued coefficients of the form
\begin{align}
  P(D) u + Q(D) A u & = f \quad \text{ in }(0,2\pi)^n,\label{rd01}\\
  D^\beta u\big|_{x_j=2\pi} - e^{2\pi\nu_j} D^\beta u\big|_{x_j=0}& = 0 \quad (j=1,\dots,n,\, |\beta|<m_1).
  \label{rd02}
\end{align}
Here $P(D)$ is a partial differential operator of order $m_1$
acting on $u=u(x)$ with $x\in(0,2\pi)^n$, $Q(D)$ a partial differential operator of order $m_2 \leq m_1$, $A$ is a closed linear operator acting in a Banach space $X$, and $\nu_1,\dots,\nu_n\in\C$ are given numbers. We refer to the boundary conditions as $\nu$-periodic. Note that for $\nu_j=0$ we have periodic boundary conditions in direction $j$, whereas for $\nu_j=\frac i2$ we have antiperiodic boundary conditions in this direction. In general, we have different boundary conditions (i.e., different $\nu_j$) in different directions.

As a motivation for studying problem \eqref{rd01}-\eqref{rd02}, we want to mention two classes of problems: First, the boundary value problem
\eqref{rd01}-\eqref{rd02} includes equations of the form
\begin{equation}\label{rd03}
   u_{t}(t)  + A u (t)  = f (t) \quad (t\in (0,2\pi))
\end{equation}
and
\begin{equation}\label{rd04}
	u_{tt}(t) - a A u_t(t) - \alpha Au(t) = f(t) \quad (t \in (0,2\pi))
\end{equation}
with periodic or $\nu$-periodic boundary conditions. Equations of the form \eqref{rd03} and \eqref{rd04} were considered in
\cite{Arendt-Bu-2002} and \cite{Keyantuo-Lizama-2006}, respectively. These equations fit into our context
 by taking $n=1$, $P(D) = \partial_t$ and $Q(D) = 1$ for \eqref{rd03} and
 $P(D)=\partial_t^2$, $Q(D)=-a\partial_t - \alpha$ for
 \eqref{rd04}.

As a second motivation for studying \eqref{rd01}-\eqref{rd02}, we consider a boundary value problem of cylindrical type where the domain is of the form $\Omega = (0,2\pi)^n\times V$ with $V\subset\R^{n_V}$ being a  sufficiently smooth domain with compact boundary. The operator is assumed to split in the sense that
\begin{equation}
  \label{rd05}
  \mathcal A(x,D) = P(x^1,D_1) + Q(x^1,D_1) A_V(x^2,D_2)
\end{equation}
where the differential operators $P(x^1,D_1)$ and $Q(x^1,D_1)$ act on $x^1\in(0,2\pi)^n$ only and the differential operator $A_V(x^2,D_2)$ acts on $x^2\in V$ only. The boundary conditions are assumed to be $\nu$-periodic in $x^1$-direction, whereas in $V$ the operator $A_V(x^2,D_2)$ of order $2m_V$ may be supplemented with general boundary conditions $B_1(x^2,D_2),\dots$, $B_{m_V}(x^2,D_2)$. The simplest example of such an operator is the Laplacian in a finite cylinder $(0,2\pi)^n\times V$ with $\nu$-periodic boundary conditions in the cylindrical directions and Dirichlet boundary conditions on $(0,2\pi)^n\times\partial V$.

Our first main result (Theorem~\ref{maintheorem_ell}) gives, under appropriate assumptions on $P$, $Q$, and $A$, equivalent conditions for the unique solvability of \eqref{rd01}-\eqref{rd02} in $L^p$-Sobolev spaces. This results generalizes results from \cite{Arendt-Bu-2002} and \cite{Keyantuo-Lizama-2006} on equations \eqref{rd03} and \eqref{rd04}, respectively.

In particular in connection with operators of the form \eqref{rd05} in cylindrical domains, one is also interested in parabolic theory. Therefore, in Section~5 we study problems of the form
\begin{equation}
  \label{rd06}
  \begin{aligned}
    u_t + \mathcal A(x,D) u & = f \quad (t\in [0,T],\, x\in (0,2\pi)^n\times V),\\
    B_j(x,D) u & = 0 \quad(t\in [0,T],\, x\in (0,2\pi)^n\times \partial V,\, j=1,\dots,m_V),\\
(D^\beta u)|_{x_j = 2\pi} -  e^{2\pi \nu_j} (D^\beta u)|_{x_j = 0} & =0
	\quad (j = 1,\ldots,n;\ |\beta| < m_1),\\
u(0,x) & = u_0(x) \quad (x\in (0,2\pi)^n\times V).
     \end{aligned}
\end{equation}
Here $\mathcal A(x,D)$ is of the form \eqref{rd05}. If $(A_V,B_1,\dots,B_{m_V})$ is a parabolic boundary value problem in the sense of parameter-ellipticity (see \cite[Section 8]{Denk-2003}), we obtain, under suitable assumptions on $P$ and $Q$, maximal $L^q$-regularity for \eqref{rd06} (see Theorems~\ref{mainresult_1} and \ref{mainresult_2} below). The proof of maximal regularity is based on the $\mathcal R$-boundedness of the resolvent related to \eqref{rd06}.

Apart from its own interest, the consideration of $\nu$-periodic boundary conditions also allows us to address boundary conditions of mixed type. As the simplest example, when $a = 0$ we can analyze equation~\eqref{rd04} with Dirichlet-Neumann type boundary conditions
\[ u(0) = 0,\; u_t(\pi) =0.\]
The connection to periodic and antiperiodic boundary conditions is given by suitable extensions of the solution. This was also considered in \cite{Arendt-Bu-2002} where -- starting from periodic boundary conditions -- the pure Dirichlet and the pure Neumann case could be treated.

The main tool to address problems \eqref{rd01}-\eqref{rd02} and \eqref{rd06} is the theory of discrete vector-valued Fourier multipliers. Taking the Fourier series in the cylindrical directions, we are faced with the question under which conditions an operator-valued Fourier series defines a bounded operator in $L^p$. This question was answered by Arendt and Bu in \cite{Arendt-Bu-2002} for the one-dimensional case $n=1$, where  a  discrete operator-valued Fourier multiplier result for UMD spaces
and applications to periodic Cauchy problems of first and second order
in Lebesgue- and H\"older-spaces can be found. For general $n$, the main result on vector-valued Fourier multipliers is contained in  \cite{Bu-Kim-2004}.
A shorter proof of this result by means of induction
based on the result for $n = 1$ in \cite{Arendt-Bu-2002} is given in \cite{Bu-2006}.
As pointed out by the authors in \cite{Arendt-Bu-2002} and \cite{Bu-Kim-2004}, the results
can as well be deduced from \cite[Theorems 3.7, 3.8]{Strkalj-Weis-2007}.

A generalization of the results in \cite{Arendt-Bu-2002}
to periodic first order integro-differential equations
in Lebesgue-, Besov- and H\"older-spaces
is given in \cite{Keyantuo-Lizama-2004}.
Here the concept of 1-regularity in the context of sequences
is introduced (see Remark~\ref{rem1reg} below).

In \cite{Keyantuo-Lizama-2006} one finds a comprehensive treatment of periodic
second order differential equations of type \eqref{rd04}
 in Lebesgue- and H\"older-spaces.
In particular, the special case of a Cauchy problem of
second order, i.e. $\alpha = 0, a = 1$,
where $A$ is the generator of a strongly continuous cosine function
is investigated.
In \cite{Keyantuo-Lizama-Poblete-2009} more general equations are treated
in the mentioned spaces as well as in Triebel-Lizorkin-spaces. Moreover,
applications to nonlinear equations are presented.

Maximal regularity of  second order initial value problems  of the type
\begin{align*}
	u_{tt}(t) + B u_t(t) + Au(t) &= f(t) \quad (t \in [0,T)),\\
\ u(0) = u_t(0) & = 0
\end{align*}
is treated in \cite{Chill-Sri-2005} and  \cite{Chill-Sri-2008}. In particular,
 $p$-independence of maximal regularity for this type of second order problems
 is shown. The same equation involving
dynamic boundary conditions is studied in \cite{Xiao-Liang-2004}.
The non-autonomous second order problem,
involving $t$-dependent operators $B(t)$ and $A(t)$, is treated in
\cite{Batty-Chill-Sri-2008}. We also refer to \cite{Xiao-Liang-1998} for the treatment of higher order Cauchy problems.

In \cite{Arendt-Rabier-2009} various properties as e.g. Fredholmness
of the operator $\partial_t - A(\cdot)$ associated to the
non-autonomous periodic first order Cauchy-problem in $L^p$-context
are investigated. Results on this operator based on Floquet theory
are obtained in the PhD-thesis \cite{Gauss-2001}. We remark that in Floquet theory
$\nu$-periodic (instead of periodic) boundary conditions appear in a natural way.

For the treatment of boundary value problems in $(0,1)$ with operator-valued coefficients subject to
numerous types of homogeneous and inhomogeneous boundary conditions, we refer to
\cite{Favin-Labbas-Maingot-Tanabe-Yagi-2008}, \cite{Favini-Shakmurov-Yakubov-2009}, \cite{Favini-Yakubov-2010} and the references therein.
Their approaches mainly rely on semigroup theory and do not allow for an easy generalization to $(0,1)^n$.
In \cite{Favini-Shakmurov-Yakubov-2009} however, applications to boundary value problems in the cylindrical space domain
$(0,1) \times V$ can be found.

The usage of operator-valued multipliers to treat cylindrical in space boundary value problems
was first carried out in \cite{Guidotti-2004} and \cite{Guidotti-2005} in a Besov-space setting.
In these papers the author constructs semiclassical fundamental
solutions for a class of elliptic operators on infinite cylindrical domains $\mathbb R^n \times V$.
This proves to be a strong tool for the treatment of related
elliptic and parabolic (\cite{Guidotti-2004} and \cite{Guidotti-2005}), as well as
of hyperbolic (\cite{Guidotti-2005}) problems. Operators in cylindrical domains with a similar splitting property as in the present paper were, in the case of an infinite cylinder, also considered in \cite{Nau-Saal-2011}.

%%%%%%%%%%%%%%%%%%%%%%%%%%%%%%%%%%%%%%%%%%%%%%%%%%%%%%%%%%%%%%%%%%
\section{Discrete Fourier multipliers and $\mathcal R$-boundedness}
%%%%%%%%%%%%%%%%%%%%%%%%%%%%%%%%%%%%%%%%%%%%%%%%%%%%%%%%%%%%%%%%%%
In the following, let $X$ and $Y$ be Banach spaces, $1<p<\infty$, $n\in\N$, and $\mathcal Q_n
:= (0,2\pi)^n$. By $\mathcal L(X,Y)$ we denote the space of all bounded linear operators
from $X$ to $Y$, and we set $\mathcal L(X) := \mathcal L(X,X)$.  By $L^p(\mathcal Q_n,X)$ we denote the standard Bochner space of $X$-valued $L^p$-functions defined on $\mathcal Q_n$. For $f \in L^p(\mathcal Q_n,X)$ and $\mathbf k \in \mathbb Z^n$ the $\mathbf k$-th Fourier coefficient of $f$ is given by
\begin{equation}\label{rd07}
	\hat f (\mathbf k) := \frac{1}{(2\pi)^n}
	\int_{\mathcal Q_n}e^{-i \mathbf{k}\cdot x}f(x) dx\,.
\end{equation}
%%%%%%%%%%%%%%%%%%%%%%%%%%%%%%%%%%%%%%%%%%%%%%%%%%%%%%%%%%%%%%%%%%
By  Fejer's Theorem %\cite[Thm. 4.2.19]{Arendt-Batty-Hieber-Neubrander-2001},
we see that  $f(x) = 0$ almost everywhere
if $\hat f (\mathbf k) = 0$ for all $\mathbf k \in \mathbb Z^n$
as well as $f(x) = \hat f (\mathbf 0)$ almost everywhere
if $\hat f (\mathbf k) = 0$ for all $\mathbf k \in \mathbb Z^n \setminus \{\mathbf 0\}$.
Moreover for $f, g \in L^p(\mathcal Q_n,X)$ and a closed operator $A$ in $X$
it holds that
$f(x) \in D(A)$ and $Af(x) = g(x)$ almost everywhere
if and only if
$\hat f(\mathbf k) \in D(A)$ and
$A \hat f(\mathbf k) = \hat g(\mathbf k)$ for all $\mathbf k \in \mathbb Z^n$.
We will frequently make use of these observations without further comments.
%%%%%%%%%%%%%%%%%%%%%%%%%%%%%%%%%%%%%%%%%%%%%%%%%%%%%%%%%%%%%%%%%%
\begin{definition}
A function $M\colon \mathbb Z^n \rightarrow \mathcal L(X,Y)$
is called a (discrete) $L^p$-multiplier if
for each $f \in L^p(\mathcal Q_n,X)$ there exists a $g \in L^p(\mathcal Q_n,Y)$ such that
	\[
		\hat g (\mathbf k) = M(\mathbf k) \hat f (\mathbf k)
		\quad ( \mathbf k \in \mathbb Z^n).
	\]
In this case there exists a unique operator
$T_M \in \mathcal L (L^p(\mathcal Q_n,X),L^p(\mathcal Q_n,Y))$ associated to $M$ such that
\begin{equation}\label{Fejer1}
	 (T_M f)\hat ~ (\mathbf k) = M(\mathbf k) \hat f (\mathbf k)
	\quad ( \mathbf k \in \mathbb Z^n).
\end{equation}
\end{definition}
The property of being a  Fourier multiplier is closely related to the concept of $\mathcal R$-boundedness. Here we give only the definition and some properties which will be used later on; as   references for $\mathcal R$-boundedness we mention \cite{Kunstmann-Weis-2004} and \cite{Denk-2003}.
%%%%%%%%%%%%%%%%%%%%%%%%%%%%%%%%%%%%%%%%%%%%%%%%%%%%%%%%%%%%%%%%%%
\begin{definition}
A family $\mathcal{T} \subset \mathcal{L}(X,Y)$ is called
 $\mathcal{R}$-bounded if there exist a $C > 0$ and a $p \in [1,
\infty)$ such that for all $N \in \mathbb{N}$, $T_j \in \mathcal{T}$,
$x_j \in X$ and all independent symmetric $\{-1,1\}$-valued random variables
$\varepsilon_j$ on a probability space $(\Omega,\mathcal{A},P)$ for
$j =1,...,N$, we have that
\begin{equation}\label{R_bound}
\Big\| \sum\limits_{j=1}^{N} \varepsilon_j T_jx_j\Big\|_{L^p(\Omega,Y)} \leq C_p
\Big\|
\sum\limits_{j=1}^{N} \varepsilon_j x_j\Big\|_{L^p(\Omega,X)}.
\end{equation}
The smallest $C_p>0$ such that \eqref{R_bound} is satisfied is
called $\mathcal R_p$-bound of $\mathcal T$ and denoted by $\mathcal
R_p(\mathcal T)$.
\end{definition}
By Kahane's inequality, \eqref{R_bound} holds for all $p\in[1,\infty)$ if it holds for one $p\in[1,\infty)$. Therefore, we will drop the $p$-dependence of $\mathcal R_p(\mathcal T)$ in the notation and write $\mathcal R(\mathcal T)$.
%%%%%%%%%%%%%%%%%%%%%%%%%%%%%%%%%%%%%%%%%%%%%%%%%%%%%%%%%%%%%%%%%%
\begin{lemma}\label{HintSummeKahane}
a) Let  $Z$ be a third Banach space and let $\mathcal{T},
\mathcal{S} \subset \mathcal{L}(X,Y)$ as well as $\mathcal{U} \subset
\mathcal{L}(Y,Z)$ be $\mathcal{R}$-bounded.
Then
$\mathcal{T} + \mathcal{S}$,
$\mathcal{T} \cup \mathcal{S}$ and
$\mathcal{U} \mathcal{T}$
are $\mathcal{R}$-bounded as well
and we have
$$\mathcal{R}(\mathcal{T} + \mathcal{S}),\
\mathcal{R}(\mathcal{T} \cup \mathcal{S}) \leq
\mathcal{R}(\mathcal{S})+\mathcal{R}(\mathcal{T}), \quad
\mathcal{R}(\mathcal{U}  \mathcal{T}) \leq
\mathcal{R}(\mathcal{U})\mathcal{R}(\mathcal{T}).$$
Furthermore, if $\overline{\mathcal T}$ denotes the closure of
$\mathcal T$ with respect to the strong operator topology, then we have
$\mathcal R(\overline{\mathcal T})=\mathcal R(\mathcal T)$.

b) {\em Contraction principle of Kahane:}\
Let $p \in[1,\infty)$. Then for all $N\in\mathbb{N}, x_j\in X,
\varepsilon_j$ as above, and for all $a_j, b_j\in\mathbb{C}$ with
$|a_j|\leq|b_j|$ for $j=1,\dots,N$ we have
\begin{equation}
\Big\|\sum_{j=1}^N a_j\varepsilon_jx_j\Big\|_{L^p(\Omega,X)}\leq
2\Big\|\sum_{j=1}^Nb_j\varepsilon_jx_j\Big\|_{L^p(\Omega,X)}.
\end{equation}
\end{lemma}
%%%%%%%%%%%%%%%%%%%%%%%%%%%%%%%%%%%%%%%%%%%%%%%%%%%%%%%%%%%%%%%%%%
For $M\colon \mathbb Z^n \rightarrow \mathcal L(X,Y)$ and $1\le j\le n$
we inductively define the differences (discrete derivatives)
\[
  \Delta_j^\ell M(\mathbf k) := \Delta_j^{\ell-1} M(\mathbf k) - \Delta_j^{\ell-1} M(\mathbf k - \mathbf e_j) \quad (\ell\in\N,\,\mathbf k\in\Z^n),
\]
where $\mathbf e_j$ denotes the $j$-th unit vector in $\R^n$ and where we have set
 $\Delta^0_j M(\mathbf k) := M(\mathbf k)\;(\mathbf k\in\Z^n)$.
As $\Delta_i^{\gamma_i}$ and $\Delta_j^{\gamma_j}$ commute for $1 \leq i,j\leq n$,
for a multi-index $\gamma\in\N_0^n$ the expression
\[
	\Delta^\gamma M(\mathbf k)
	:= \big(\Delta_1^{\gamma_1} \cdots \Delta_n^{\gamma_n} M \big)(\mathbf k)
\quad(\mathbf k\in\Z^n)
\]
is well-defined. Given $\alpha, \beta, \gamma \in \mathbb N_0^n$, we will
 write $\alpha \leq \gamma \leq \beta$ if
$\alpha_j \leq \gamma_j \leq \beta_j$ for all $1\leq j \leq n$. We also set
 $|\alpha|:=\alpha_1+\dots + \alpha_n$, $\mathbf 0 := (0,\dots, 0)$, and $\mathbf 1:= (1,\dots, 1)$.

We recall that a Banach space $X$ is called a UMD space or a Banach space of class $\mathcal{HT}$ if there exists a $q\in(1,\infty)$ (equivalently: if for all $q\in(1,\infty)$) the Hilbert transform defines a bounded operator in $L^q(\R,X)$. A Banach space $X$ is said to have property $(\alpha)$ if there exists a $C>0$ such that for all $N\in\N$, $\alpha_{ij}\in\C$ with $|\alpha_{ij}|\le 1$, all $x_{ij}\in X$, and all independent symmetric $\{+1,-1\}$-valued random variables $\varepsilon_i^{(1)}$ on a probability space $(\Omega_1,\mathcal A_1,P_1)$ and $\varepsilon_j^{(2)}$ on a probability space $(\Omega_2,\mathcal A_2,P_2)$ for $i,j=1,\dots,N$ we have
\[ \Big\|\sum_{i,j=1}^N \alpha_{ij}\varepsilon_i^{(1)} \varepsilon_j^{(2)} x_{ij}\Big\|_{L^2(\Omega_1\times\Omega_2,X)} \le C \Big\|\sum_{i,j=1}^N \varepsilon_i^{(1)} \varepsilon_j^{(2)} x_{ij}\Big\|_{L^2(\Omega_1\times\Omega_2,X)}.\]
The following result from Bu and Kim characterizes discrete Fourier multipliers by $\mathcal R$-boundedness.
%%%%%%%%%%%%%%%%%%%%%%%%%%%%%%%%%%%%%%%%%%%%%%%%%%%%%%%%%%%%%%%%%%
\begin{theorem}[\cite{Bu-Kim-2004}]\label{Bu}
\mbox{}
a) Let $X, Y$ be UMD spaces
and let $\mathcal T \subset \mathcal L(X,Y)$ be $\mathcal R$-bounded.
If $M\colon \mathbb Z^n \rightarrow \mathcal L(X,Y)$ satisfies
\begin{align}\label{mult_cond_1}
	\big\{|\mathbf k|^{|\gamma|}
	\Delta^\gamma M(\mathbf k):
	\mathbf k \in \mathbb Z^n,\ \mathbf 0 \leq \gamma \leq \mathbf 1 \big\} \subset \mathcal T,
\end{align}	
then $M$ defines a Fourier multiplier.

b) If $X, Y$ additionally enjoy property $(\alpha)$, then
\begin{align}\label{mult_cond_2}
	\big\{\mathbf k^{\gamma}
	\Delta^\gamma M(\mathbf k):\
	\mathbf k \in \mathbb Z^n,\ \mathbf 0 \leq \gamma \leq \mathbf 1 \big\} \subset \mathcal T
\end{align}	
is sufficient. In this case the set
\[\{T_{M}: \text{ M satisfies condition \eqref{mult_cond_2}} \} \subset \mathcal L
\big( L^p(\mathcal Q_n,X), L^p(\mathcal Q_n,Y)\big)\]
is $\mathcal R$-bounded again.
\end{theorem}
%%%%%%%%%%%%%%%%%%%%%%%%%%%%%%%%%%%%%%%%%%%%%%%%%%%%%%%%%%%%%%%%%%
\begin{remark}
In \cite{Bu-Kim-2004}, Theorem \ref{Bu} is stated with
discrete derivatives $\tilde \Delta$ defined in
such a way that $\Delta^\gamma M(\mathbf k + \gamma) = \tilde \Delta^\gamma M(\mathbf k)$.
However, as for fixed $\gamma\in\{0,1\}^n$ there exist $c,C > 0$ such that
$c|\mathbf k - \gamma| \leq |\mathbf k | \leq C|\mathbf k - \gamma|$
for $\mathbf k  \in\Z^n\setminus\{0,1\}^n$, Lemma \ref{HintSummeKahane} shows
our formulation to be equivalent to the one in \cite{Bu-Kim-2004}.
Throughout this article, we will make use of this estimate frequently
without any further comment.
\end{remark}
%%%%%%%%%%%%%%%%%%%%%%%%%%%%%%%%%%%%%%%%%%%%%%%%%%%%%%%%%%%%%%%%%%

The following lemma
states some properties for discrete derivatives, where
$(S_\mathbf k)_{\mathbf k\in\Z^n}$ and $(T_\mathbf k)_{\mathbf k\in\Z^n}$ denote arbitrary commuting sequences in $\mathcal L (X)$.
For $\alpha \in \mathbb N_0^n \setminus \{\mathbf 0\}$, let
\[
	\mathcal{Z}_{\alpha} := \Big\{\mathcal{W} = (\omega^1,\dots,\omega^{r});\
	1 \leq r \leq |\alpha|,\ \mathbf 0 \le \omega^j \leq \alpha,\, \omega^j\not=\mathbf 0,\, \sum_{j = 1}^{r}
	\omega^j = \alpha \Big\}
\]
denote the set of all additive decompositions of $\alpha$ into
$r = r_\mathcal{W}$  multi-indices and set
$\mathcal{Z}_{\mathbf 0} := \{\emptyset\}$ and $r_{\emptyset}:=0$.
For $\mathcal W \in \mathcal Z_\alpha$ we set $\omega^*_j := \sum_{l = j + 1}^r \omega_l$.
In the following,  $c_{\alpha,\beta}$ and $c_\mathcal W$ will denote integer constants  depending on $\alpha, \beta$ and $\mathcal W$, respectively.
%%%%%%%%%%%%%%%%%%%%%%%%%%%%%%%%%%%%%%%%%%%%%%%%%%%%%%%%%%%%%%%%%%
\begin{lemma}\label{tec_lemma}
a) \emph{Leibniz rule:}\ For $\alpha\in\N_0^n$ we have
\begin{align*}
	\Delta^\alpha(ST)_\mathbf k
	 = \sum\limits_{\mathbf 0\le\beta \leq \alpha} c_{\alpha,\beta} (\Delta^{\alpha - \beta}S)_{\mathbf k - \beta} (\Delta^\beta T)_\mathbf k\quad(\mathbf k\in\Z^n).
\end{align*}
b) Let $(S^{-1})_\mathbf k := (S_\mathbf k)^{-1}$ exist for all $\mathbf k\in\Z^n$. Then, for $\alpha\in\N_0^n$ we have
\begin{align*}
	\Delta^\alpha (S^{-1})_\mathbf k & = \sum\limits_{\mathcal W \in \mathcal Z_\alpha} c_\mathcal W (S^{-1})_{\mathbf k - \alpha}
			\prod\limits_{j = 1}^{r_\mathcal W} (\Delta^{\omega_j}S)_{\mathbf k - \omega^*_j}  (S^{-1})_{\mathbf k - \omega^*_j}\quad(\mathbf k\in\Z^n).
\end{align*}
\end{lemma}
%%%%%%%%%%%%%%%%%%%%%%%%%%%%%%%%%%%%%%%%%%%%%%%%%%%%%%%%%%%%%%%%%%
\begin{proof}
We will show both assertions by induction on $|\alpha|$, the case $|\alpha|=0$ being obvious.

a) By definition, we have
\[
	\Delta^{\mathbf e_j}(ST)_\mathbf k = (ST)_\mathbf k - (ST)_{\mathbf k - \mathbf e_j}
	= S_{\mathbf k - \mathbf e_j} (\Delta^{\mathbf e_j} T)_\mathbf k + (\Delta^{\mathbf e_j} S)_\mathbf k T_\mathbf k,
\]
and for $\alpha' := \alpha - \mathbf e_j$ where $\alpha_j \neq 0$ we obtain
\begin{align*}
	\Delta^{\alpha}(ST)_\mathbf k & =
	\Delta^{\mathbf e_j}\sum\limits_{\beta \leq \alpha'} c_{\alpha'\beta}
	(\Delta^{\alpha' - \beta}S)_{\mathbf k - \beta} (\Delta^{\beta} T)_\mathbf k\\
	&=  \sum\limits_{\beta \leq \alpha'} c_{\alpha'\beta}
	\Big( (\Delta^{\alpha'  - \beta}S)_{\mathbf k - (\beta + \mathbf e_j)} (\Delta^{\beta + \mathbf e_j} T)_\mathbf k
	+	(\Delta^{\alpha' + \mathbf e_j - \beta}S)_{\mathbf k - \beta} (\Delta^{\beta} T)_\mathbf k \Big)\\
	&=  \sum\limits_{\beta \leq \alpha} c_{\alpha\beta} (\Delta^{\alpha - \beta}S)_{\mathbf k - \beta} (\Delta^{\beta} T)_\mathbf k.
\end{align*}

b) For $|\alpha|\ge 1$, we apply a)  to $S S^{-1}$ and get
\begin{align*}
	0 = \Delta^{\alpha}(S S^{-1})_\mathbf k & =  S_{\mathbf k - \alpha} (\Delta^\alpha S^{-1})_\mathbf k
	+ \sum\limits_{\beta < \alpha} c_{\alpha\beta} (\Delta^{\alpha - \beta}S)_{\mathbf k - \beta} (\Delta^{\beta} S^{-1})_\mathbf k.
\end{align*}
Hence
\begin{align*}
	(\Delta^\alpha & S^{-1})_\mathbf k  =
	- (S^{-1})_{\mathbf k - \alpha}  \sum\limits_{\beta < \alpha} c_{\alpha\beta}
		(\Delta^{\alpha - \beta}S)_{\mathbf k - \beta} (\Delta^{\beta} S^{-1})_\mathbf k\\
	& = - \sum\limits_{\beta < \alpha} \sum\limits_{\mathcal W \in \mathcal Z_\beta} c_\mathcal W
		(S^{-1})_{\mathbf k - \alpha}  (\Delta^{\alpha - \beta}S)_{\mathbf k - \beta} (S^{-1})_{\mathbf k - \beta}
			\prod\limits_{j = 1}^{r_\mathcal W} (\Delta^{\omega_j}S)_{\mathbf k - \omega^*_j}  (S^{-1})_{\mathbf k - \omega^*_j}\\
	& = \sum\limits_{\mathcal W \in \mathcal Z_\alpha} c_\mathcal W
		(S^{-1})_{\mathbf k - \alpha}  (\Delta^{\omega_1}S)_{\mathbf k - \omega_1^*} (S^{-1})_{\mathbf k - \omega_1^*}
			\prod\limits_{j = 2}^{r_\mathcal W} (\Delta^{\omega_j}S)_{\mathbf k - \omega^*_j}  (S^{-1})_{\mathbf k - \omega^*_j}.
\end{align*}
\end{proof}
%%%%%%%%%%%%%%%%%%%%%%%%%%%%%%%%%%%%%%%%%%%%%%%%%%%%%%%%%%%%%%%%%%
\begin{definition}
Consider a polynomial $P: \mathbb R^n \rightarrow \mathbb C;\ \xi \mapsto P(\xi)$
and let $P^\#$ denote its principal part.

a) $P$ is called elliptic if $P^\#(\xi) \neq 0$ for $\xi \in \mathbb R^n \setminus \{\mathbf 0\}$.

b) Let $\phi\in(0,\pi)$ and let $\Sigma_\phi:= \{\lambda\in\C\setminus\{0\}: |\arg(\lambda)| <
\phi\}$ be the open sector with angle $\phi$.
Then $P$ is called parameter-elliptic in $\overline{\Sigma}_{\pi-\phi}$ if
$\lambda + P^\#(\xi) \neq 0$ for $(\lambda,\xi) \in \overline{\Sigma}_{\pi-\phi} \times \mathbb R^n \setminus \{(0,\mathbf 0)\}$. In this case,
\[\varphi_P:= \inf\{\phi\in(0,\pi): P \textrm{ is parameter-elliptic in }\overline{\Sigma}_{\pi-\phi}\}\]
is called the angle of parameter-ellipticity of $P$.
\end{definition}
%%%%%%%%%%%%%%%%%%%%%%%%%%%%%%%%%%%%%%%%%%%%%%%%%%%%%%%%%%%%%%%%%%
\begin{remark}\label{rem_ellptic}
a) By quasi-homogeneity of $(\lambda,\xi)\mapsto \lambda+P^\#(\xi)$, we easily see that
$P$ is parameter-elliptic in
$\overline{\Sigma}_{\pi-\phi}$ if
and only if for all polynomials $N$ with $\deg N \leq \deg P$ there exist $C > 0$ and
a bounded subset $G \subset \mathbb R^n$ such that the estimate
$|\xi|^m |N(\xi)| \leq C |\lambda + P(\xi)|$ holds
for all $\lambda \in \overline{\Sigma}_{\pi-\phi}$, all $0 \leq m \leq \deg P - \deg N$
and all $\xi \in \mathbb R^n \setminus G$.

b) In the same way,
$P$ is elliptic if and only if the assertion in a) is valid for $\lambda = 0$.

c) By induction, one can see that for $|\alpha| \le \deg P$ the discrete polynomial
$\Delta^\alpha P(\mathbf k)$ defines a polynomial of degree not greater than
$\deg P-|\alpha|$. If $P$ is elliptic, this implies $|\mathbf k|^{|\alpha|}|\Delta^{\alpha}
P(\mathbf k)|\le C |P(\mathbf k)|\;(\mathbf k\in\Z^n\setminus G)$ with a finite set $G\subset\Z^n$.
\end{remark}
%%%%%%%%%%%%%%%%%%%%%%%%%%%%%%%%%%%%%%%%%%%%%%%%%%%%%%%%%%%%%%%%%%
\begin{proposition}\label{1regPolyMult}
Let $A$ be a closed linear operator in a UMD space $X$.
Consider polynomials
$P,Q: \mathbb Z^n \rightarrow \mathbb C$
such that
\begin{itemize}
\item $P$ and $Q$ are  elliptic,
\item $\big(P(\mathbf k) + Q(\mathbf k)A \big)^{-1}$ exists for all $\mathbf k \in \mathbb Z^n$,
\item $\big \{ P(\mathbf k) \big(P(\mathbf k) + Q(\mathbf k) A \big)^{-1}:
			\mathbf k \in \mathbb Z^n \big\}$ is $\mathcal R$-bounded.
\end{itemize}
Then for every polynomial $N$ with
$\deg N \leq \deg P$ the map
\[
	M: \mathbb Z^n \rightarrow \mathcal L(X): \quad
	\mathbf k \mapsto N(\mathbf k)
	\big(P(\mathbf k) + Q(\mathbf k) A \big)^{-1}
\]
defines an $L^p$-multiplier for $1 < p < \infty$.
\end{proposition}
\begin{proof}
Lemma \ref{tec_lemma} yields
\begin{align*}
	|\mathbf k|^{|\gamma|} & \Delta^\gamma M (\mathbf k)  = \sum\limits_{\beta \leq \gamma}
			\sum\limits_{\mathcal W \in \mathcal Z_\beta} c_\mathcal W
			|\mathbf k|^{|\gamma - \beta|}(\Delta^{\gamma - \beta} N) (\mathbf k - \beta)\big(P(\mathbf k - \beta) + Q(\mathbf k - \beta) A \big)^{-1}\\
	& \cdot  \prod\limits_{j = 1}^{r_\mathcal W} |\mathbf k|^{|\omega_j|}
			\big( \Delta^{\omega_j} P(\mathbf k - \omega^*_j) + \Delta^{\omega_j}Q(\mathbf k - \omega^*_j) A \big)
			\big(P(\mathbf k - \omega^*_j) + Q(\mathbf k - \omega^*_j) A \big)^{-1}.
\end{align*}
By Remark~\ref{rem_ellptic}, we know that $\deg(\Delta^{\gamma-\beta}N)\le \deg N-|\gamma-\beta|$. This and the ellipticity of $P$ imply $|\mathbf k|^{|\gamma-\beta|}|\Delta^{\gamma-\beta}N(\mathbf k)|\le C|P(\mathbf k)|$ for $\mathbf k\in\Z^n\setminus G$ with a finite set $G\subset\Z^n$. By Kahane's contraction principle, we obtain the $\mathcal R$-boundedness of
\[ \Big\{ |\mathbf k|^{|\gamma-\beta|}\Delta^{\gamma-\beta}N(\mathbf k-\beta)\big( P(\mathbf k-\beta)+
Q(\mathbf k-\beta)A\big)^{-1}: \mathbf k\in\Z^n\setminus G\Big\}.\]
Since
\[
	Q(\mathbf k)A\big(P(\mathbf k) + Q(\mathbf k)A\big)^{-1} =
	\textrm{id}_X - P(\mathbf k) \big(P(\mathbf k) + Q(\mathbf k)A\big)^{-1},
\]
in the same way the $\mathcal R$-boundedness of
\[ \Big\{ |\mathbf k|^{|\omega_j|}\Delta^{\omega_j}Q(\mathbf k- \omega^*_j)A
\big( P(\mathbf k- \omega^*_j)+
Q(\mathbf k- \omega^*_j)A\big)^{-1}: \mathbf k\in\Z^n\setminus G\Big\}\]
follows from the ellipticity of $Q$.
Now the assertion follows from Lemma~\ref{HintSummeKahane} and Theorem~\ref{Bu}.
\end{proof}
%%%%%%%%%%%%%%%%%%%%%%%%%%%%%%%%%%%%%%%%%%%%%%%%%%%%%%%%%%%%%%%%%%
Proposition \ref{1regPolyMult} is closely related to the concept of 1-regularity of
complex-valued sequences, introduced in \cite{Keyantuo-Lizama-2004} for the one dimensional case $n = 1$.
In fact, if $Q(\mathbf k) \neq 0$ for all
$\mathbf k \in \mathbb Z^n$, we may write $M(\mathbf k)= \frac{N(\mathbf k)}{Q(\mathbf k)}\big(\frac{P(\mathbf k)}{Q(\mathbf k)} + A\big)^{-1}$.
Hence, for $n = 1$ we enter the framework of \cite[Proposition 5.3]{Keyantuo-Lizama-Poblete-2009}, i.e.
$M(k) = a_k(b_k - A)^{-1}$ with $(a_k)_{k \in \mathbb Z}, (b_k)_{k \in \mathbb Z} \subset \mathbb C$. We will give a generalization of this concept to arbitrary $n$ and briefly indicate the connection to the results above.
%%%%%%%%%%%%%%%%%%%%%%%%%%%%%%%%%%%%%%%%%%%%%%%%%%%%%%%%%%%%%%%%%%
\begin{definition}\label{1regular}
We call a pair of sequences
$(a_\mathbf k, b_\mathbf k)_{\mathbf k \in \mathbb Z^n} \subset \mathbb C^2$
 {$1$-regular} if for all
$\mathbf 0 \leq \gamma \leq \mathbf 1$
there exist a finite set $K\subset\Z^n$ and a constant $C > 0$
such that
\begin{equation}\label{1reg1}
 |\mathbf k^\gamma| \max \{ |(\Delta^\gamma a)_\mathbf k|, |(\Delta^\gamma b)_\mathbf k|\} \leq C |b_\mathbf k|
 \quad(\mathbf k\in\Z^n\setminus K).
\end{equation}
We say the pair $(a_\mathbf k, b_\mathbf k)_{\mathbf k \in \mathbb Z^n}$
is {strictly $1$-regular} if
$|\mathbf k^\gamma|$ can be replaced by $|\mathbf k|^{|\gamma|}$ in
\eqref{1reg1}. A  sequence $(a_\mathbf k)_{\mathbf k \in \mathbb Z^n}$
is called (strictly) 1-regular if $(a_\mathbf k,a_\mathbf k)_{\mathbf k \in \mathbb Z^n}$
has this property.
\end{definition}
%%%%%%%%%%%%%%%%%%%%%%%%%%%%%%%%%%%%%%%%%%%%%%%%%%%%%%%%%%%%%%%%%%
\begin{remark}\label{rem1reg}
a) In the case $n = 1$,
a sequence $(a_k)_{k \in \mathbb Z}\subset\mathbb C\setminus\{0\}$
 is 1-regular in $\mathbb Z$ in the sense of Definition \ref{1regular} if and only if
the sequence
$ \big( \frac{k (a_{k+1} - a_{k})}{a_{k}} \big)_{k \in \mathbb Z}$
is bounded. Hence our definition extends the one from \cite{Keyantuo-Lizama-2004}
for a sequence $(a_k)_{k \in \mathbb Z}$.

b) With $\gamma = 0$ the definition especially requests
$|a_\mathbf k| \leq C |b_\mathbf k|$ for $\mathbf k \in \mathbb Z^n \setminus K$.

c) Strict 1-regularity implies 1-regularity. If $n = 1$ both concepts are equivalent.

d) Subject to the assumptions of Proposition \ref{1regPolyMult}, let
$Q(\mathbf k) \neq 0$ for $\mathbf k \in\Z^n$. Then
the pair
$\big(\frac{N(\mathbf k)}{Q(\mathbf k)},\frac{P(\mathbf k)}{Q(\mathbf k)}\big)_{\mathbf k \in \mathbb Z^n}$
is strictly 1-regular.

e) Again from Lemma \ref{tec_lemma} we deduce the following variant of Proposition \ref{1regPolyMult}:
Let $b_\mathbf k \in \rho(A)$ for all $\mathbf k \in \mathbb Z^n$, let
$\mathcal R (\{b_\mathbf k (b_\mathbf k - A)^{-1}:\ \mathbf k \in \mathbb Z^n \setminus G\}) < \infty$
for some finite subset $G \subset \mathbb Z^n$, and let $(a_\mathbf k, b_\mathbf k)_{\mathbf k \in \mathbb Z^n}$
be strictly 1-regular.
Then $M(\mathbf k) := a_\mathbf k (b_\mathbf k - A)^{-1}$ defines a Fourier multiplier.
\end{remark}
%%%%%%%%%%%%%%%%%%%%%%%%%%%%%%%%%%%%%%%%%%%%%%%%%%%%%%%%%%%%%%%%%%

%%%%%%%%%%%%%%%%%%%%%%%%%%%%%%%%%%%%%%%%%%%%%%%%%%%%%%%%%%%%%%%%%%
\section{$\nu$-periodic boundary value problems}
%%%%%%%%%%%%%%%%%%%%%%%%%%%%%%%%%%%%%%%%%%%%%%%%%%%%%%%%%%%%%%%%%%
\begin{definition}
Let $X$ be a Banach space,  $m \in \mathbb N_0$, $n \in \mathbb N$ and $\nu \in \mathbb C^n$.
We set $D^\alpha:=D_1^{\alpha_1}\ldots D_n^{\alpha_n}$ with $D_j=-i\frac\partial{\partial j}$ and denote by $W_{\nu, per}^{m,p}(\mathcal Q_n,X)$ the space of all
$u \in W^{m,p}(\mathcal Q_n,X)$ such that
for all $j \in \{1, \ldots, n\}$ and all $|\alpha| < m$
it holds that
\[
	(D^\alpha u)|_{x_j = 2\pi} = e^{2\pi \nu_j} (D^\alpha u)|_{x_j = 0}.
\]
For sake of convenience we set $W_{per}^{m,p}(\mathcal Q_n,X) := W_{0, per}^{m,p}(\mathcal Q_n,X)$.
\end{definition}

We give some helpful characterizations of the space $W^{m,p}_{\nu, per}(\mathcal Q_n,X)$
where we omit the rather simple proof.
%%%%%%%%%%%%%%%%%%%%%%%%%%%%%%%%%%%
\begin{lemma}\label{Wnuper}
The following assertions are equivalent:
\begin{itemize}
\item[(i)]	$u \in W_{\nu,per}^{m,p}(\mathcal Q_n,X)$.
\item[(ii)]	$u \in W^{m,p}(\mathcal Q_n,X)$ and
						for all $|\alpha| \leq m$ it holds that
						\[
							(e^{-\nu \cdot}D^\alpha u)\hat~ (\mathbf k)
							= (\mathbf k - i\nu)^\alpha (e^{-\nu \cdot}u)\hat~ (\mathbf k)
						\]
						for all $\mathbf k \in \mathbb Z^n$.
\item[(iii)]There exists $v \in  W_{per}^{m,p}(\mathcal Q_n,X)$ such that
						$u = e^{\nu \cdot} v$.
\end{itemize}
\end{lemma}
%%%%%%%%%%%%%%%%%%%%%%%%%%%%%%%%%%%
The following lemma characterizes multipliers such that the associated
operators map
$L^p(\mathcal Q_n,X)$ into $W_{per}^{\alpha,p}(\mathcal Q_n,X)$.
The proof follows the one for the case $n = 1$ of \cite[Lemma~2.2]{Arendt-Bu-2002}.
%%%%%%%%%%%%%%%%%%%%%%%%%%%%%%%%%%%
\begin{lemma}\label{MultWper}
Let $1 \leq p < \infty$, $m \in \mathbb N$ and
$M\colon \mathbb Z^n \rightarrow \mathcal L(X)$.
Then the following assertions are equivalent:
\begin{itemize}
\item[(i)]	$M$ is an $L^p$-multiplier such that the associated
						operator $T_M \in \mathcal L(L^p(\mathcal Q_n,X))$ maps
						$L^p(\mathcal Q_n,X)$ into $W_{per}^{m,p}(\mathcal Q_n,X)$.
\item[(ii)]	$M_\alpha\colon \mathbb Z^n \rightarrow \mathcal L(X),\
						\mathbf k \mapsto \mathbf k^\alpha M(\mathbf k)$
						is an $L^p$-multiplier for all $|\alpha| = m$.
\end{itemize}
\end{lemma}

Let $X$ be a UMD space  and $A$ be
a closed linear operator  in $X$. With $n \in \mathbb N$ and $\nu \in \mathbb C^n$
we consider the boundary value problem in $\mathcal Q_n$ given by
\begin{equation}\label{RWP1}
\begin{array}{r@{\quad = \quad}l@{\quad}l}
\mathcal A(D)u  & f & \quad (x\in \mathcal Q_n), \\
(D^\beta u)|_{x_j = 2\pi} -  e^{2\pi \nu_j} (D^\beta u)|_{x_j = 0} & 0 &
	\quad(j = 1,\ldots,n;\ |\beta| < m_1).\\
\end{array}
\end{equation}
In view of the boundary conditions, we refer to the boundary
value problem \eqref{RWP1}
as $\nu$-periodic. Here
\[
	\mathcal A(D) := P(D) + Q(D)A
		:=  \sum_{|\alpha|\leq m_1} p_\alpha D^\alpha
						+ \sum_{|\alpha|\leq m_2} q_\alpha D^\alpha A
\]
with $m_1,m_2 \in \mathbb N$, $m_2 \leq m_1$, and
$p_\alpha, q_\alpha \in \mathbb C$.
In what follows, with $m := m_1$ we frequently write
$\mathcal A(D) = \sum_{|\alpha|\leq m} (p_\alpha D^\alpha + q_\alpha D^\alpha A )$
where additional coefficients $q_\alpha$, that is, where $m_2 < |\alpha| \leq m_1$, are understood to be equal to zero.
Besides that we define the complex polynomials
$P(z) := \sum_{|\alpha| \leq m_1} p_\alpha z^\alpha$ and
$Q(z) := \sum_{|\alpha| \leq m_2} q_\alpha z^\alpha$ for $z \in \mathbb C^n$.

\begin{definition}
A solution of the boundary value problem \eqref{RWP1} is understood as a function
$u \in W^{m_1,p}_{\nu,per}(\mathcal Q_n,X) \cap W^{m_2,p}(\mathcal Q_n,D(A))$
such that $\mathcal A(D)u(x) = f(x)$ for almost every $x \in \mathcal Q_n$.
\end{definition}

\begin{remark}
Since the trace operator with respect to one direction
and tangential derivation commute, the $\nu$-periodic boundary conditions
as imposed in \eqref{RWP1} are equivalent
to
\[
	(D_j^\ell u)|_{x_j = 2\pi} -  e^{2\pi \nu_j} (D_j^\ell u)|_{x_j = 0} = 0
	\quad (
	j = 1,\ldots,n,\; 0 \leq \ell < m_1).
\]
\end{remark}

\begin{theorem}\label{maintheorem_ell}
Let $1 < p <\infty$, and  assume $P$ and $Q$ to be elliptic.
Then the following assertions are equivalent:
\begin{itemize}
\item [(i)]		For each $f \in L^p(\mathcal Q_n,X)$ there exists a unique solution of \eqref{RWP1}.	
\item [(ii)]	$\big( P(\mathbf k - i\nu) + Q(\mathbf k - i\nu) A \big)^{-1} \in \mathcal L(X)$ exists	
							for $\mathbf k \in \mathbb Z^n$ and
							\[
								M_\alpha(\mathbf k)	:= \mathbf k^\alpha \big( P(\mathbf k - i\nu) + Q(\mathbf k - i\nu) A \big)^{-1}
							\]
							defines a Fourier multiplier for every $|\alpha| = m_1$.
\item [(iii)]	$\big( P(\mathbf k - i\nu) + Q(\mathbf k - i\nu) A \big)^{-1} \in \mathcal L(X)$ exists
							for $\mathbf k \in \mathbb Z^n$ and for all $|\alpha| = m_1$ there exists
							a finite subset $G \subset \mathbb Z^n$ such that the sets
							$
								\{M_\alpha(\mathbf k);\ \mathbf k \in \mathbb Z^n \setminus G\}
							$
							are $\mathcal R$-bounded.
\end{itemize}
\end{theorem}
\begin{proof}
(i) $\Rightarrow$ (ii): Let $f \in L^p(\mathcal Q_n,X)$ be arbitrary
and let $u$
be a solution of \eqref{RWP1} with right-hand side $e^{\nu \cdot}f$.
Then $e^{-\nu \cdot}\mathcal A(D) u = f$.

To compute the Fourier coefficients, we first remark that
\[ \big( e^{-\nu\cdot} P(D) u\big)\hat ~ (\mathbf k) = P(\mathbf k-i\nu)(e^{-\nu\cdot} u)\hat~
(\mathbf k)\]
by Lemma~\ref{Wnuper}.
Concerning $e^{-\nu\cdot} Q(D) Au$, note that
by definition of a solution we have $Au\in W^{m_2,p}(\mathcal Q_n,X)$. Due to the closedness of $A$, we obtain $D^\alpha Au = AD^\alpha u$ for $|\alpha|\le m_2$, and consequently $Au\in W^{m_2,p}_{\nu,per}(\mathcal Q_n,X)$. Now we can apply Lemma~\ref{Wnuper} to see
\[ \big(e^{-\nu\cdot}Q(D)Au\big)\hat ~ (\mathbf k) = Q(\mathbf k-i\nu)(e^{-\nu\cdot} Au)\hat ~ (\mathbf k) = Q(\mathbf k-i\nu) A (e^{-\nu\cdot}u)\hat ~ (\mathbf k).
\]
Writing $\mathbf k_\nu := \mathbf k - i\nu$ for short, we obtain
\begin{equation}\label{Fequation}
\big( P(\mathbf k_\nu )
+  Q(\mathbf k_\nu ) A \big) \big( e^{-\nu \cdot}u\big) \hat ~ (\mathbf k)  = \hat f (\mathbf k).
\end{equation}

For arbitrary $y \in X$ and $\mathbf k \in \mathbb Z^n$, the choice $f := e^{i\mathbf k \cdot} y$ shows
$\big( P(\mathbf k_\nu )
+  Q(\mathbf k_\nu ) A \big)$ to be surjective.
Let $z \in D(A)$ such that
$\big( P(\mathbf k_\nu )
+  Q(\mathbf k_\nu ) A \big) z = 0$.
For fixed $\mathbf k \in \mathbb Z^n$ set $v := e^{i\mathbf k \cdot} z$
and $u := e^{\nu \cdot}v$. %u is \nu-periodic
Then
%\begin{align*}
%\big( P(\mathbf k_\nu ) +  Q(\mathbf k_\nu ) A \big) \hat v (\mathbf k) = 0,
%\end{align*}
%hence
\begin{align*}
P(\mathbf k_\nu )  \big( e^{-\nu \cdot}u \big)\hat ~ (\mathbf k)
+  Q(\mathbf k_\nu )A \big( e^{-\nu \cdot}u \big)\hat ~ (\mathbf k) = 0.
\end{align*}
As $(e^{-\nu\cdot} u)\hat~(\mathbf m)=0$ for all $\mathbf m\not=\mathbf k$, this gives
 $\mathcal A(D) u = 0$, hence $v = u = 0$ and $z = 0$.

Altogether we have shown bijectivity of $ P(\mathbf k_\nu )
+  Q(\mathbf k_\nu ) A $ for $\mathbf k \in \mathbb Z^n$. The  closedness of $A$
yields $\big( P(\mathbf k_\nu )
+  Q(\mathbf k_\nu ) A \big)^{-1} \in \mathcal L(X)$.

For $f \in L^p(\mathcal Q_n,X)$ let
$u$
be a solution
of \eqref{RWP1} with right hand side $e^{\nu \cdot}f$ and
$v := e^{-\nu \cdot}u$.
Then $v\in W^{m_1,p}_{per}(\mathcal Q_n,X)$, and \eqref{Fequation} implies
\begin{align*}%\label{hatresolvent}
\hat v(\mathbf k) =
\big( P(\mathbf k_\nu ) + Q(\mathbf k_\nu ) A \big)^{-1}
\hat f(\mathbf k).
\end{align*}
This shows
\[
	M_0: \mathbb Z^n \rightarrow \mathcal L(L^p(\mathcal Q_n,X));\
	\mathbf k \mapsto\big( P(\mathbf k_\nu ) + Q(\mathbf k_\nu ) A \big)^{-1}
\]
to be a Fourier multiplier such that $T_{M_0}$ maps $L^p(\mathcal Q_n,X)$
into
$W^{m_1,p}_{per}(\mathcal Q_n,X)$.
Due to Lemma \ref{MultWper}, we have that $M_\alpha$ is a Fourier multiplier
for all $|\alpha| = m_1$.

(ii) $ \Rightarrow $ (iii): This follows as in \cite[Prop. 1.11]{Arendt-Bu-2002}.

(iii) $ \Rightarrow $ (i):
For $\mathbf k \neq \mathbf 0$ it holds that
\begin{align*}
 	P(\mathbf k_\nu ) &  \big( P(\mathbf k_\nu ) + Q(\mathbf k_\nu ) A \big)^{-1} \\
 	&	= \frac{P(\mathbf k_\nu )}{\sum\limits_{j = 1}^n \mathbf k^{m_1 \mathbf e_j}}
		\bigg( \sum\limits_{j = 1}^n \mathbf k^{m_1 \mathbf e_j}
		\big( P(\mathbf k_\nu ) + Q(\mathbf k_\nu ) A \big)^{-1} \bigg)
\end{align*}
and as there exists $C > 0$ such that
$|P(\mathbf k_\nu )| \leq C |\sum_{j = 1}^n \mathbf k^{m_1 \mathbf e_j}|$
for $\mathbf k \in \mathbb Z^n \setminus G$ with suitably chosen finite  $G\subset\Z^n$,
Lemma \ref{HintSummeKahane} shows
that the set
\[
 \left\{
 P(\mathbf k_\nu )  \big( P(\mathbf k_\nu ) + Q(\mathbf k_\nu ) A \big)^{-1}:\
	\mathbf k \in \mathbb Z^n \setminus G
	\right\}
\]
is $\mathcal R$-bounded as well.
By Proposition \ref{1regPolyMult} it follows
that $M_\alpha$ for $|\alpha| = m_1$ as well as $P(\cdot - i \nu)M_0$
%the right-hand side of \eqref{P_multiplier}
are Fourier multipliers.
For  arbitrary $f \in L^p(\mathcal Q_n,X)$
we therefore get
$v :=  T_{M_0} (e^{-\nu \cdot}f) \in W^{m_1,p}_{per}(\mathcal Q_n,X)$.
As
\begin{equation}\label{QA_multiplier}
\begin{aligned}
Q(\mathbf k_\nu )A
& \big( P(\mathbf k_\nu ) + Q(\mathbf k_\nu ) A \big)^{-1}\\
& = %\frac{ Q(\mathbf k_\nu )}{Q(\mathbf k_\nu )}
	\textrm{id}_X - P(\mathbf k_\nu )\big( P(\mathbf k_\nu ) + Q(\mathbf k_\nu ) A \big)^{-1},
\end{aligned}
\end{equation}
$Q(\cdot - i \nu)A M_0$ is a Fourier multiplier, too. By ellipticity of $Q$ and Lemma~\ref{HintSummeKahane} again, the same holds for $\mathbf k^\alpha A\big(P(\mathbf k_\nu ) + Q(\mathbf k_\nu )A\big)^{-1}$, $|\alpha|\le m_2$.

Set $u := e^{\nu \cdot}v = e^{\nu \cdot} T_{M_0} e^{-\nu \cdot} f$. Then
$u$ solves \eqref{RWP1} by construction, and Lemma~\ref{MultWper} yields $u \in W^{m_1,p}_{\nu,per}(\mathcal Q_n,X)$ and  $Au \in W^{m_2,p}_{\nu,per}(\mathcal Q_n,X)$.
Finally, uniqueness of $u$  follows immediately from the uniqueness of the representation
as a Fourier series.
\end{proof}
%%%%%%%%%%%%%%%%%%%%%%%%%%%%%%%%%%%%%%%%%%%%%%%%%%%%%%%%%%%%%%%%%%
\begin{remark}\label{Q_elliptic}
We have seen in the proof that if one  of the equivalent conditions in Theorem \ref{maintheorem_ell}
is satisfied, we have  $Au \in W^{m_2,p}_{\nu,per}(\mathcal Q_n,X)$.
In particular, we get
\[
(D^\beta A u)|_{x_j = 2\pi} -  e^{2\pi \nu_j} (D^\beta A u)|_{x_j = 0} = 0
\quad (j = 1,\ldots,n;\ |\beta| < m_2)
\]
as additional boundary conditions in \eqref{RWP1}.
\end{remark}

%%%%%%%%%%%%%%%%%%%%%%%%%%%%%%%%%%%%%%%%%%%%%%%%%%%%%%%%%%%%%%%%%%
Theorem \ref{maintheorem_ell} enables us to treat
Dirichlet-Neumann type boundary conditions on
$\tilde{\mathcal Q}_n := (0,\pi)^n$ for symmetric operators,
provided $P$ and $Q$ are of appropriate structure.
More precisely, we call a differential operator $\mathcal A(D) = \sum_{|\alpha|\leq m} (p_\alpha D^\alpha + q_\alpha D^\alpha A) $
symmetric if for all $|\alpha| \leq m$  either $p_\alpha = q_\alpha = 0$
or $\alpha \in 2 \mathbb N_0^n$. In particular, $m_1$ is even. As  examples, the operators
$\mathcal A(D_t) := D^2_t + A$ and $\mathcal A(D_1,D_2) := (D_1^2 + D_2^2)^2 + (D_1^4 + D_2^4) A $
are symmetric and satisfy the conditions on $P$ and $Q$ from Theorem~\ref{maintheorem_ell}.

In each direction $j\in\{1,\dots,n\}$, we will consider  one of the following  boundary
conditions:

\begin{itemize}
  \item[(i)] $D_j^\ell u|_{x_j=0} = D_j^\ell u|_{x_j=\pi}=0\quad(\ell=0,2,\dots,m_1-2)$,
  \item[(ii)] $D_j^\ell u|_{x_j=0} = D_j^\ell u|_{x_j=\pi}=0\quad(\ell=1, 3,\dots,m_1-1)$,
  \item[(iii)] $D_j^\ell u|_{x_j=0} = D_j^{\ell+1} u|_{x_j=\pi}=0\quad(\ell=0,2,\dots,m_1-2)$,
  \item[(iv)] $D_j^{\ell+1} u|_{x_j=0} = D_j^\ell u|_{x_j=\pi}=0\quad(\ell=0,2,\dots,m_1-2)$.
 \end{itemize}
Note that for a second-order operator, (i) is of Dirichlet type, (ii) is of Neumann type, and (iii) and (iv) are of mixed type. For instance, in case (iii) we have $u|_{x_j=0} = 0$ and $D_j u|_{x_j=\pi} =0$. Therefore, we refer to these boundary conditions as conditions of Dirichlet-Neumann type. Note that the types may be different in different directions.

%%%%%%%%%%%%%%%%%%%%%%%%%%%%%%%%%%%%%%%%%%%%%%%%%%%%%%%%%%%%%%%%%%
\begin{theorem}\label{maintheorem_ell_DN}
Let $\mathcal A(D)$ be symmetric, with $P$ and $Q$ being elliptic, and let the boundary conditions be of Dirichlet-Neumann type as explained above. Define $\nu\in\mathbb C^n$ by setting $\nu_j:= 0$ in cases (i) and (ii) and $\nu_j:= i/2$ in cases (iii) and (iv). If for this $\nu$ one of the equivalent conditions of Theorem~\ref{maintheorem_ell} is fulfilled, then
 for each $f \in L^p(\tilde{ \mathcal Q}_n, X)$ there exists a unique solution
$u \in W^{2m,p}(\tilde{ \mathcal Q}_n,X)$ of $\mathcal A(D) u = f$ satisfying the boundary conditions.
\end{theorem}
%%%%%%%%%%%%%%%%%%%%%%%%%%%%%%%%%%%%%%%%%%%%%%%%%%%%%%%%%%%%%%%%%%
\begin{proof}
Following an idea from \cite{Arendt-Bu-2002}, the solution is constructed by a suitable even or odd extension of the right-hand side from $(0,\pi)^n$ to $(-\pi,\pi)^n$. For simplicity of notation, let us consider the case   $n = 2$ and boundary conditions of type (ii) in direction $x_1$ and of type (iii) in direction $x_2$. By definition, this leads to $\nu_1=0$  and $\nu_2 = \frac{i}{2}$.

Let $f \in L^p(\tilde{ \mathcal Q}_n,X)$ be arbitrary. First considering the even
extension of $f$ to the rectangle $(-\pi,\pi) \times (0,\pi)$ and afterwards its odd extension
to $(-\pi,\pi) \times (-\pi,\pi)$,
we end up with a function $F$ which fulfills
$F(x_1,x_2) =  F(- x_1,x_2)$
as well as $F(x_1,x_2) = - F(x_1,-x_2)$ a.e. in
$(-\pi,\pi)^2$.

Now we can apply Theorem~\ref{maintheorem_ell}  with $\nu = (\nu_1,\nu_2)^T$ as above.
(Here and in the following, the result of Theorem~\ref{maintheorem_ell} has to be shifted
from the interval $(0,2\pi)^n$ to the interval $(-\pi,\pi)^n$.) This yields
 a unique solution $U$ of
\begin{equation}\label{RWP_symmetric}
\begin{array}{r@{\quad = \quad}l@{\quad}l}
\mathcal A(D)U  &  F  & \text{in } (-\pi,\pi) \times (-\pi,\pi), \\
 		D_1^\ell U |_{x_1 = -\pi}  & D_1^\ell U |_{x_1 = \pi} &  (  \ell = 0,\dots, m_1-1), \\
 		-D_2^\ell U |_{x_2 = -\pi}  & D_2^\ell U |_{x_2 = \pi} &  (  \ell = 0,\dots, m_1-1).\\
\end{array}
\end{equation}

Symmetry of $\mathcal A(D)$ now shows that $V_1(x_1,x_2) :=  U(- x_1,x_2)$ and
 $V_2(x_1,x_2) :=  - U(x_1,- x_2)$ $(x\in(-\pi,\pi)^2)$ are  solutions
of \eqref{RWP_symmetric} as well.
By uniqueness, $ V_1 = U = V_2$ follows.

Hence $U_{x_2} := U(\cdot,x_2) \in W^{m,p}((-\pi,\pi),X) \subset C^{m-1}((-\pi,\pi),X)$
for a.e. $x_2 \in (-\pi,\pi)$ is even. Together with symmetry of $U_{x_2}$ due to \eqref{RWP_symmetric}, this yields
\[
	U^{(\ell)}_{x_2}(0) = U^{(\ell)}_{x_2}(\pi) = 0\quad (\ell=1, 3, \dots, \textstyle{m_1-1.}).
\]
Accordingly for a.e. $x_1 \in (-\pi,\pi)$ we have that $U_{x_1}$
is odd, and antisymmetry due to \eqref{RWP_symmetric} gives
\[
	U^{(\ell)}_{x_1}(0) = U^{(\ell + 1)}_{x_1}(\pi) = 0\quad (\ell=0, 2, \dots, \textstyle{m_1-2}).
\]
Therefore, $u:= U|_{(0,\pi)^n}$ solves $\mathcal A(D) u = f$ with boundary conditions (ii) for $j=1$ and (iii) for $j=2$.

For arbitrary $n \in \mathbb N$ and arbitrary boundary conditions of Dirichlet-Neumann type, the construction of the solution follows the same lines. Here we choose even extensions in the cases (ii) and (iv) and odd extensions in the cases (i) and (iii).

On the other hand, let $u$ be a solution of $\mathcal A(D) u = f$ satisfying boundary conditions of Dirichlet-Neumann type.
We extend $u$ and $f$ to $U$ and $F$ on $(-\pi,\pi)^n$
as described above. Then $U \in W^{m,p}((-\pi,\pi)^n,X)$,
$Q(D)AU \in L^p((-\pi,\pi)^n,X)$ and due to symmetry of $\mathcal A(D)$ we see that,
apart from a shift, $U$ solves \eqref{RWP1}
with right-hand side $F$ and $\nu$ defined as above. Thus, uniqueness of $U$ yields
uniqueness of $u$ and the proof is complete.
\end{proof}
%%%%%%%%%%%%%%%%%%%%%%%%%%%%%%%%%%%%%%%%%%%%%%%%%%%%%%%%%%%%%%%%%%
\begin{remark}
In case $n = 1$ ellipticity of $P$ does no longer force $P$ to be of even order.
Hence the same results can be achieved if $\mathcal A(D)$ is antisymmetric
in the obvious sense, e.g. $\mathcal A(D_t) := D_t^3 + D_t + D_t A$.
\end{remark}

%%%%%%%%%%%%%%%%%%%%%%%%%%%%%%%%%%%%%%%%%%%%%%%%%%%%%%%%%%%%%%%%%%
\section{Maximal regularity of cylindrical boundary value problems
with $\nu$-periodic boundary conditions}
%%%%%%%%%%%%%%%%%%%%%%%%%%%%%%%%%%%%%%%%%%%%%%%%%%%%%%%%%%%%%%%%%%
Let $F$ be a UMD space and
let $\Omega := \mathcal Q_n \times V \subset \mathbb R^{n + n_V}$
with $V \subset \mathbb R^{n_V}$. For $x \in \Omega$ we write
$x = (x^1,x^2) \in \mathcal Q_n  \times V$, whenever we want to refer
to the cylindrical geometry of $\Omega$. Accordingly, we write
$\alpha = (\alpha^1,\alpha^2) \in \mathbb N_0^n \times \mathbb N _0^{n_V}$
for a multiindex $\alpha\in\mathbb{N}_0^{n + n_V}$ and
$D^{\alpha} = D^{(\alpha^1,\alpha^2)} =: D_1^{\alpha^1}D_2^{\alpha^2}$.

In the sequel we investigate the vector-valued parabolic initial
boundary value problem
\begin{equation}\label{Ini_BVP}
  \begin{aligned}
    u_t + \mathcal A(x,D) u & = f \quad (t\in J,\, x\in \mathcal Q_n\times V),\\
    B_j(x,D) u & = 0 \quad(t\in J,\, x\in \mathcal Q_n\times \partial V,\, j=1,\dots,m_V),\\
(D^\beta u)|_{x_j = 2\pi} -  e^{2\pi \nu_j} (D^\beta u)|_{x_j = 0} & =0
	\quad (j = 1,\ldots,n;\ |\beta| < m_1),\\
u(0,x) & = u_0(x) \quad (x\in \mathcal Q_n\times V).
     \end{aligned}
\end{equation}
Here $J := [0,T)$, $0 < T \leq \infty$, denotes a time interval, and the differential operator
$\mathcal A(x,D)$ has the form
\begin{align*}
\mathcal A(x,D) &
=P(x^1,D_1) + Q(D_1)A_V(x^2,D_2)\\
& :=  \sum_{|\alpha^1|\leq m_1} p_{\alpha^1}(x^1) D_1^{\alpha^1}
						+ \sum_{|\alpha^1|\leq m_2} q_{\alpha^1} D_1^{\alpha^1} A_V(x^2,D_2).
\end{align*}
The operator $A_V(x^2,D_2)$ is assumed to be of order $2m_V$ and is augmented with boundary conditions
\begin{align*}
 B_j(x,D) = B_j(x^2,D_2) \quad (j = 1,\ldots,m_V)
\end{align*}
with operators
$B_j(x^2,D_2)$ of order $m_j < 2m_V$ acting on the boundary of $V$.

This class of equations fits into the framework of Section~3 if we define the operator $A=A_V$ in Section~3 as the $L^p$-realization of the boundary value problem $((A_V(x^2,D_2),B_1(x^2,D_2),\dots, B_{m_V}(x^2,D_2))$. More precisely, for $1<p<\infty$ we define the operator $A_V$ in $L^p(V,F)$ by
\begin{align*}
  D(A_V) & := \{ u\in W^{2m,p}(V,F): B_j(x^2,D_2)u=0\;(j=1,\dots,m_V)\},\\
  A_Vu & := A_V(x,D) u := A_V(x^2,D_2) u \quad (u\in D(A_V)).
\end{align*}
Throughout this section, we will assume that the boundary value problem $(A_V,B_1,$ $\dots,B_{m_V})$ satisfies standard smoothness and parabolicity assumptions as, e.g., given in \cite[Theorem~8.2]{Denk-2003}. In particular, $V$ is assumed to be a domain with compact $C^{2m_V}$-boundary, and $(A_V,B_1,\dots,B_{m_V})$ is assumed to be parameter-elliptic with angle $\varphi_{A_V}\in[0,\pi)$. For the notion of parameter-ellipticity of a boundary value problem, we refer to \cite[Section~8.1]{Denk-2003}.

Recall that a sectorial operator $A$ is called $\mathcal R$-sectorial if there exists a $\theta\in (0,\pi)$ such that
\begin{equation}
  \label{rd08}
  \mathcal R\big(\{\lambda(\lambda+A)^{-1}: \lambda\in\Sigma_{\pi-\theta}\}\big)<\infty.
\end{equation}
For an $\mathcal R$-sectorial operator, $\phi_A^{\mathcal R}:= \inf\{\theta\in (0,\pi): \eqref{rd08}\;\textrm{holds}\}$ is called the $\mathcal R$-angle of $A$ (see \cite[p.~42]{Denk-2003}). The $\mathcal R$-sectoriality of an operator is closely related to maximal regularity. Recall that a closed and densely defined operator in a Banach space $X$ has maximal $L^q$-regularity if for each $f\in L^q((0,\infty),X)$ there exists a unique solution $w\colon(0,\infty)\to D(A)$ of the Cauchy problem
\begin{align*}
 w_t + Aw & = f \quad \textrm{in } (0,\infty),\\
w(0) & = 0
\end{align*}
satisfying the estimate
\[ \|w_t\|_{L^q((0,\infty),X)} + \|Aw\|_{L^q((0,\infty),X)} \le C \|f\|_{L^q((0,\infty),X)}\]
with a constant $C$ independent of $f$. By a well-known result due to Weis \cite[Thm.~4.2]{Weis-2001}, $\mathcal R$-sectoriality in a UMD space with $\mathcal R$-angle less than $\frac \pi 2$ is equivalent to maximal $L^q$-regularity for all $1<q<\infty$. In \cite{Denk-2003} it was shown that standard parameter-elliptic problems lead to $\mathcal R$-sectorial operators:

\begin{proposition}[\mbox{\cite[Theorem~8.2]{Denk-2003}}]
\label{resultdenk2003}
Under the assumptions above, for each
 $\phi > \varphi_{A_V}$ there exists
a $\delta_V = \delta_V(\phi) \geq 0$ such that $A_V + \delta_V$ is $\mathcal R$-sectorial
with $\mathcal R$-angle $\phi_{A_V+\delta_V}^{\mathcal R} \le \phi$. Moreover,
\begin{equation}\label{RZusatzV}
\mathcal{R} ( \{\lambda^{1- \frac{|\alpha^2|}{2m_V}}D^{\alpha^2} (\lambda +
A_V + \delta_V)^{-1};\ \lambda \in \Sigma_{\pi - \phi},\ 0 \leq |\alpha^2|
\leq 2m_V\} ) < \infty.
\end{equation}
\end{proposition}

We will show that under suitable assumptions on $P$ and $Q$,  $\mathcal R$-sectoriality of $A_V$ implies $\mathcal R$-sectoriality of the operator related to the cylindrical problem \eqref{Ini_BVP}. For this consider the resolvent problem corresponding to \eqref{Ini_BVP} which is given by
\begin{equation}\label{RWP3}
\begin{array}{r@{\quad=\quad}l@{\quad}l}
\lambda u + \mathcal A(x,D) u  & f & ( x\in \mathcal Q_n \times V), \\
 B_j(x,D) u &  0  & ( x\in  \mathcal Q_n \times  \partial  V,\ j = 1,\ldots, m_V),\\
(D^\beta u)|_{x_j = 2\pi} -  e^{2\pi \nu_j} (D^\beta u)|_{x_j = 0} & 0 &	(j = 1,\ldots,n,\ |\beta| < m_1).
\end{array}
\end{equation}
For sake of readability, we assume that $m_1 = 2m_V$. The $L^p(\Omega,F)$-realization of the boundary value problem \eqref{RWP3} is defined as
\begin{align*}
D(\mathbb A) & := \{u \in W^{m_1,p}(\Omega,F) \cap W^{m_1,p}_{\nu,per}(\mathcal Q_n,L^p(V,F)):\\
& \quad	\quad \quad
	 B_j(x,D) u = 0 \; (j=1,...,m_V),\quad
	 A_V(x,D) u \in W^{m_2,p}(\mathcal Q_n, L^p(V,F))\} ,\\
\mathbb A u & := \mathcal A(x,D) u \quad (u \in D(\mathbb A)).
\end{align*}
%%%%%%%%%%%%%%%%%%%%%%%%%%%%%%%%%%%%%%%%%%%%%%%%%%%%%%%%%%%%%%%%%%%%%%%%%%%%%%%
\begin{remark} a) Since $m_2 \leq m_1$ it holds that
\[
	D(\mathbb A) =
	W^{m_1,p}(\Omega,F) \cap W^{m_1,p}_{\nu,per}(\mathcal Q_n,L^p(V,F))
	\cap W^{m_2,p}(\mathcal Q_n,D(A_V)).
\]
b) The following techniques apply as well to equations with mixed orders $m_1 \neq 2m_V$.
Then, in the definition of $D(\mathbb A)$, the space $W^{m_1,p}(\Omega,F)$
has to be replaced by
$\{u \in L^p(\Omega,F):\ D^\alpha u \in L^p(\Omega,F) \text{ for }
\frac{|\alpha^1|}{m_1} + \frac{|\alpha^2|}{2m_V} \leq 1\}$.
\end{remark}

%%%%%%%%%%%%%%%%%%%%%%%%%%%%%%%%%%%%%%%%%%%%%%%%%%%%%%%%%%%%%%%%%%%%%%%%%%%%%%%
\subsection{Constant coefficients}
%%%%%%%%%%%%%%%%%%%%%%%%%%%%%%%%%%%%%%%%%%%%%%%%%%%%%%%%%%%%%%%%%%%%%%%%%%%%%%%
We first assume
$P(x^1,D_1) = P(D_1)$ and $Q(x^1,D_1)=Q(D_1)$ to have constant coefficients
and set
\[
	 \mathcal A_0 := \mathcal A_0(x^2,D) := P(D_1) +Q(D_1)\big( A_V + \delta_V \big).
\]
With $\mathbb A_0 u := \mathcal A_0 u$ for
$u \in D(\mathbb A_0) := D(\mathbb A)$ we formally get
$ (\lambda +  \mathbb A_0)^{-1}
	=  e^{\nu \cdot} T_{M_\lambda} e^{-\nu \cdot}$
where $T_{M_\lambda}$ denotes the associated operator to
\[
	M_\lambda(\mathbf k) :=
	\big( \lambda + P(\mathbf k - i\nu)+ Q(\mathbf k - i\nu) (A_V + \delta_V)	 \big)^{-1}.
\]
More generally, the Leibniz rule shows
\[
	D^{\alpha} (\lambda + \mathbb A_0)^{-1}
	= D^{\alpha} e^{\nu \cdot} T_{M_\lambda} e^{-\nu \cdot}
	= \sum\limits_{\beta \leq \alpha} g_\beta(\nu) e^{\nu \cdot} T_{M^\beta_\lambda}  e^{-\nu \cdot},
\]
where $g_\beta$ is a polynomial depending on $\beta$.
Here $T_{M^\beta_\lambda}$ denotes the associated operator to
\[
	M^\beta_\lambda(\mathbf k) :=
	\mathbf k^{\beta_1} D^{\beta_2}
	\big(\lambda + P(\mathbf k - i\nu) + Q(\mathbf k - i\nu) (A_V + \delta_V) \big)^{-1}
\]
where $\beta = (\beta_1,\beta_2)^T \leq \alpha$.
%%%%%%%%%%%%%%%%%%%%%%%%%%%%%%%%%%%%%%%%%%%%%%%%%%%%%%%%%%%%%%%%%%
\begin{theorem}\label{mainresult_1}
Let $1 < p < \infty$, let $F$ be a UMD space
enjoying property ($\alpha$),
let the boundary valued problem $(A_V,B)$
fulfill the conditions of \cite[Theorem 8.2]{Denk-2003}
with angle of parameter-ellipticity $\varphi_{A_V}$, and
let $\varphi > \varphi_{A_V}$.

For $P$ and $Q$ assume that
\begin{itemize}
	\item [(i)]		$P$ is parameter-elliptic with angle $\varphi_P \in [0,\pi)$,
	\item [(ii)]	$Q$ is elliptic,
	\item	[(iii)]	$Q(\mathbf k - i \nu) \neq 0$ for all $\mathbf k \in \mathbb Z^n$ and
				there exists $\varphi_0 > \varphi_P$ such that
				$\frac{\lambda + P(\mathbf k - i \nu)}{Q(\mathbf k - i \nu)} \in \Sigma_{\pi - \varphi}$
				holds true for all $\mathbf k \in \mathbb Z^n$ and all $\lambda \in \Sigma_{\pi-\varphi_0}$.
\end{itemize}
Then for each $\delta > 0$ the $L^p$-realization $\mathbb A_0+\delta$ of $\mathcal A_0+\delta$ is $\mathcal R$-sectorial with $\mathcal R$-angle
$\phi_{\mathbb A_0 + \delta}^{\mathcal R} \leq \varphi_0$.
Moreover, it holds that
\begin{equation}\label{RZusatzA0}
\mathcal{R} ( \{\lambda^{1- \frac{|\alpha|}{m_1}}D^{\alpha}
(\lambda + \mathbb A_0 + \delta)^{-1}:\ \lambda \in \Sigma_{\pi - \phi},
\ \alpha\in\mathbb N_0^{n + n_v},\ 0 \leq |\alpha| \leq m_1\} ) < \infty.
\end{equation}

In particular, if $\varphi_0 < \frac\pi 2$ then $\mathbb A_0+\delta$
has maximal $L^q$-regularity for every $1<q<\infty$, i.e.,
the initial-boundary value problem \eqref{Ini_BVP} is well-posed in $L^q(T, L^p(\Omega,F))$.
\end{theorem}
%%%%%%%%%%%%%%%%%%%%%%%
\begin{proof}%[Proof (of Theorem \ref{mainresult_1})]
Let $\alpha\in\mathbb{N}_0^{n + n_V}$,
$0 \leq |\alpha| \leq m_1 = 2m_V$, $\mathbf 0 \leq \beta \leq \alpha$,
$\mathbf 0 \leq \gamma \leq \mathbf 1$, and $\phi > \varphi_0$.
For sake of convenience we drop the shift of $A_V$, i.e.  we assume $\delta_V = 0$.
To prove \eqref{RZusatzA0}, for arbitrary $\delta > 0$,
we apply Lemma \ref{tec_lemma}
in order to calculate $\mathbf k^\gamma\Delta^\gamma M^\beta_{\lambda + \delta}(\mathbf k)$.
In what follows we write $\mathbf k_\nu := \mathbf k - i \nu$ for short again.
As in the proof of Proposition \ref{1regPolyMult} it suffices to
show that
\begin{equation}\label{term_1}
\begin{aligned}
	\{ \lambda^{1- \frac{|\alpha|}{m_1}} \mathbf k^\omega \Delta^\omega N(\mathbf k) D^{\beta_2}
	\big(\lambda + \delta + P(\mathbf k_\nu) + Q(\mathbf k_\nu) A_V \big)^{-1}:\
	\lambda \in \Sigma_{\pi - \phi},\ \mathbf k \in \mathbb Z^n\}
\end{aligned}
\end{equation}
with $N(\mathbf k) := \mathbf k ^{\beta_1}$ and arbitrary $\omega \leq \gamma$,
\begin{equation}\label{term_2}
\begin{aligned}
	\{ \mathbf k^\omega \Delta^{\omega}P(\mathbf k_\nu)
	\big(\lambda + \delta + P(\mathbf k_\nu) + Q(\mathbf k_\nu) A_V \big)^{-1}:\
	\lambda \in \Sigma_{\pi - \phi},\ \mathbf k \in \mathbb Z^n\},
\end{aligned}
\end{equation}
with $\mathbf 0 < \omega \leq \gamma$, and
\begin{equation}\label{term_3}
\begin{aligned}
	\{ \mathbf k^\omega \Delta^{\omega} Q (\mathbf k_\nu) A_V
	\big(\lambda + \delta + P(\mathbf k_\nu) + Q(\mathbf k_\nu) A_V \big)^{-1}:\
	\lambda \in \Sigma_{\pi - \phi},\ \mathbf k \in \mathbb Z^n\}
\end{aligned}
\end{equation}
with $\mathbf 0 < \omega \leq \gamma$ are $\mathcal R$-bounded.
Due to our assumptions and Proposition \ref{resultdenk2003}, in particular due to
 \eqref{RZusatzV}, for $0 \leq |\beta_2| \leq m_1 = 2m_V$ the set
\[
	\bigg\{ \bigg( \frac{\lambda + \delta + P(\mathbf k_\nu)}{Q (\mathbf k_\nu)}\bigg)^{1-\frac{|\beta_2|}{m_1}}
	 D^{\beta_2} \bigg( \frac{\lambda + \delta + P(\mathbf k_\nu)}{Q(\mathbf k_\nu)} + A_V \bigg)^{-1}:\
	\lambda \in \Sigma_{\pi - \phi},\ \mathbf k \in \mathbb Z^n \bigg\}
\]
is $\mathcal R$-bounded. For $\beta_2 = 0$ this yields the $\mathcal R$-boundedness of
\begin{equation}\label{P_Res}
 	\{ \big( \lambda + \delta + P(\mathbf k_\nu) \big)
	\big( \lambda + \delta + P(\mathbf k_\nu) + Q (\mathbf k_\nu) A_V \big)^{-1}:\
	\lambda \in \Sigma_{\pi - \phi},\ \mathbf k \in \mathbb Z^n \}
\end{equation}
and with it the $\mathcal R$-boundedness of
\begin{equation}\label{Q_Res}
 	\{ Q (\mathbf k_\nu) A_V
	\big( \lambda + \delta + P(\mathbf k_\nu) + Q (\mathbf k_\nu) A_V \big)^{-1}:\
	\lambda \in \Sigma_{\pi - \phi},\ \mathbf k \in \mathbb Z^n \}.
\end{equation}
In particular, $Q(D)A_V (\lambda + \delta + A_0)^{-1} f \in L^p(\Omega,F))$
for $f \in L^p(\Omega,F))$.

Since $ \lambda +  P(\mathbf k_\nu) \neq 0$ for  $\lambda \in  \Sigma_{\pi - \phi}$ by condition (iii),
for each finite set $G \subset \mathbb Z^n$ there exists $C > 0$ such that
$\big|\mathbf k^\omega \Delta^{\omega}P(\mathbf k_\nu) \big| \leq C \big| \lambda + \delta + P(\mathbf k_\nu) \big|$
uniformly in $\lambda \in  \Sigma_{\pi - \phi}$ and $\mathbf k \in G$.
Together with parameter-ellipticity of $P$ and Remark \ref{rem_ellptic} this allows
to apply the contraction principle of Kahane to prove \eqref{term_2}.

Similarly, ellipticity of $Q$ proves \eqref{term_3} as well as
$D^{\alpha}A_V (\lambda + A_0)^{-1} f \in L^p(\Omega,F))$ for $|\alpha|\le m_2$.

To prove \eqref{term_1} first note that
there exists $C_\phi > 0$ such that
$|\lambda| \leq C_\phi |\lambda + \delta|$ holds for all $\lambda \in \Sigma_{\pi - \phi}$.
As $|\beta_2| \leq |\alpha|$, with the same arguments as above, it suffices to show the existence of a finite set $G \subset \mathbb Z^n$
and $C > 0$ such that
\begin{align*}
	|\lambda + \delta|^{1-\frac{|\alpha|}{m_1}}
	\bigg| \frac{\mathbf k^\omega \Delta^\omega N (\mathbf k)}{Q(\mathbf k_\nu)} \bigg|
	\leq C
	\bigg|
			\frac{\lambda + \delta + P(\mathbf k_\nu)}{Q(\mathbf k_\nu)}
	\bigg|^{1-\frac{|\beta_2|}{m_1}}
\end{align*}
holds independently of $\lambda \in \Sigma_{\pi - \phi}$
and $\mathbf k \in \mathbb Z^n \setminus G$.
Again by ellipticity of $Q$ there exists $C > 0$ such that
$ \frac{|\mathbf k_\nu|^{m_2}}{|Q(\mathbf k_\nu)|}  \leq C$.
Thus,
it is sufficient to show
\begin{align*}
	|\lambda + \delta|^{1-\frac{|\alpha|}{m_1}}
	|\mathbf k^\omega \Delta^\omega N (\mathbf k)| |\mathbf k_\nu|^{-\frac{m_2}{m_1}|\beta_2|}
	\leq C
	 \big| \lambda + \delta + P(\mathbf k_\nu) \big|^{1-\frac{|\beta_2|}{m_1}}.
\end{align*}
The polynomial $\Delta^\omega N$ has degree no larger than
$|\beta_1 - \omega|$ if $\omega \leq \beta_1$ and we have
$\Delta^\omega N \equiv 0$ else.
If $\mathbf k \neq \mathbf 0$,
%by the triangle inequality,
it is sufficient to consider
$\big|M (\mathbf k) \big| |\mathbf k|^{-\frac{m_2 }{m_1}|\beta_2|}$
instead of $	|\mathbf k^\omega \Delta^\omega N (\mathbf k)||\mathbf k_\nu|^{-\frac{m_2}{m_1}|\beta_2|}$,
where $M(\mathbf k)$ denotes a monomial of degree no larger than $|\beta_1|$.
Hence, it remains to prove
\begin{align*}
	|\lambda + \delta|^{1-\frac{|\alpha|}{m_1}} |M (\mathbf k)| |\mathbf k|^{-\frac{m_2 }{m_1}|\beta_2|}
	\leq C
	 \big| \lambda + \delta + P(\mathbf k_\nu) \big|^{1-\frac{|\beta_2|}{m_1}}.
\end{align*}
Therefore, we end up with a left-hand side that is $(m_1,1)$-quasi-homogeneous in $(\lambda + \delta,\mathbf k)$
of order no larger than $m_1 - |\alpha| + |\beta_1| - \frac{m_2 }{m_1}|\beta_2| \leq m_1 - |\beta_2|$.
%where equality holds if $m_2 = 0$ and $\beta = \alpha$.
Thus, parameter-ellipticity of $P$
yields existence of a finite set $G \subset \mathbb Z^n$ such that
\begin{align*}
	|\lambda + \delta|^{1-\frac{|\alpha|}{m_1}} |M (\mathbf k)| |\mathbf k|^{-\frac{m_2 }{m_1}|\beta_2|}
	\leq C
	 \big| \lambda + \delta + P(\mathbf k_\nu) \big|^{1-\frac{|\beta_2|}{m_1}}
\end{align*}
for $\lambda \in \Sigma_{\pi - \phi}$ and $\mathbf k \in \mathbb Z^n \setminus G$.

%Finally $ B_j(\cdot,D)(\lambda + A_0)^{-1} f = 0$ for every trigonometric
%polynomial $f$ and by density as well
%for $f \in L^p(\Omega,F)$.

The last claim on maximal $L^q$-regularity now follows from \cite[Thm.~4.2]{Weis-2001}.
\end{proof}
%%%%%%%%%%%%%%%%%%%%%%%%%%%%%%%%%%%%%%%%%%%%%%%%%%%%%%%%%%%%%%%%%%
\begin{remark}\label{Qconst}
a) Since $Q$ is elliptic, there exists a finite set $G \subset \mathbb Z^n$, such that
$Q(\mathbf k - i \nu) \neq 0$ for $\mathbf k \in \mathbb Z^n \setminus G$.
Instead of the stronger condition (iii), assume that there
exists $\varphi_0 > \varphi_P$ such that
				\begin{itemize}
					\item 	$\lambda + P(\mathbf k - i \nu) \neq 0$ for $\mathbf k \in G$ and $\lambda \in \Sigma_{\pi-\varphi_0}$ and
					\item	$\frac{\lambda + P(\mathbf k - i \nu)}{Q(\mathbf k - i \nu)} \in \Sigma_{\pi - \varphi}$
												for $\mathbf k \in \mathbb Z^n \setminus G$ and $\lambda \in \Sigma_{\pi-\varphi_0}$.
				\end{itemize}
Set $D(\tilde {\mathbb A}_0) :=  W^{m_1,p}_{\nu,per}(\mathcal Q_n,L^p(V,F)) \cap W^{m_2,p}(\mathcal Q_n,D(A_V))$.
Then for each $\delta > 0$ the $L^p$-realization $\tilde{\mathbb A}_0+\delta$ of $\mathcal A_0+\delta$
is $\mathcal R$-sectorial with $\mathcal R$-angle
$\phi_{\tilde{\mathbb A}_0 + \delta}^{\mathcal R} \leq \varphi_0$.\\[2mm]
b) Let $P = P^\#$ be given as homogeneous polynomial, let $Q \equiv 1$, $\nu  = i\mathbf r$ with $\mathbf r \in \mathbb R^n$,
and $\varphi_P + \varphi_{A_V} < \pi$.
Then for each $\varphi_0 > \max\{\varphi_P, \varphi_{A_V} \}$
condition (i) implies that there exists $\varphi > \varphi_{A_V}$ such that condition (iii) holds true.
\end{remark}
%%%%%%%%%%%%%%%%%%%%%%%
\begin{proof}
a) First note that $\big( \lambda + \delta + P(\mathbf k_\nu) + Q (\mathbf k_\nu) A_V \big)^{-1}$
still exists for all $\mathbf k \in \mathbb Z^n$.
In view of the domain of definition of $\tilde {\mathbb A}_0$,
equation \eqref{term_1} only has to be considered with
$\beta_2 = \mathbf 0$ and $\alpha = \beta_1$.
Moreover, the terms of \eqref{term_3} only appear in the formular for
$\mathbf k^\gamma \Delta^\gamma M^\beta_{\lambda + \delta}(\mathbf k)$
if $\mathbf k \in \mathbb Z^n \setminus G$. Now the proof copies. \\[2mm]
b) Since $\varphi_P + \varphi_{A_V} < \pi$, the claim follows readily due to homogeneity and parameter-ellipticity of $P = P^\#$.
Note that $\mathbf k_\nu = \mathbf k + \mathbf r \in \mathbb R^n$ and $\mathbf k_\nu = \mathbf 0$
if and only if $\mathbf r = \mathbf k = \mathbf 0$.
\end{proof}
%%%%%%%%%%%%%%%%%%%%%%%
\begin{remark}
We have seen in the proof that  $A_V u \in W^{m_2,p}_{\nu,per}(\mathcal Q_n,L^p(V,F))$, i.e.
the solution $u$ of
\eqref{RWP3} fulfills the further boundary condition
\[
	(D^\beta A_V u)|_{x_j = 2\pi} -  e^{2\pi \nu_j} (D^\beta A_V u)|_{x_j = 0} = 0
	\quad	(j = 1,\ldots,n;\ |\beta| < m_2)
\]
(cf.  Remark~\ref{Q_elliptic}). Additionally, we have seen in the proof that
\begin{equation}\label{RZusatzA0_A_V}
\mathcal{R} ( \{D^{\alpha}A_V
(\lambda + \mathbb A_0 + \delta)^{-1}:\ \lambda \in \Sigma_{\pi - \phi},\ 0 \leq |\alpha| \leq m_2\} ) < \infty.
\end{equation}
\end{remark}
%%%%%%%%%%%%%%%%%%%%%%%%%%%%%%%%%%%%%%%%%%%%%%%%%%%%%%%%%%%%%%%%%%
\begin{remark}\label{Cor_DN}
Consider again boundary value  problems in $(0,\pi)^n \times V$ with Dirichlet-Neumann type boundary conditions
and a symmetric setting with respect to $(0,\pi)^n$. As the extension and restriction operators defined above are bounded,
 Theorem \ref{maintheorem_ell_DN} immediately
yields the related result for Dirichlet-Neumann type boundary conditions. In particular, we obtain
maximal regularity results also for boundary conditions of mixed type (iii) and (iv).
\end{remark}
%%%%%%%%%%%%%%%%%%%%%%%%%%%%%%%%%%%%%%%%%%%%%%%%%%%%%%%%%%%%%%%%%%

%%%%%%%%%%%%%%%%%%%%%%%%%%%%%%%%%%%%%%%%%%%%%%%%%%%%%%%%%%%%%%%%%%%%%%%%%%%%%%%
\subsection{Non-constant coefficients of $P$}
%%%%%%%%%%%%%%%%%%%%%%%%%%%%%%%%%%%%%%%%%%%%%%%%%%%%%%%%%%%%%%%%%%%%%%%%%%%%%%%
In this subsection,
$P(x^1,D_1)$ is allowed to have non-constant coefficients,
where we assume that
\begin{equation}\label{asscoeff1}
\left\{
\begin{aligned}
p_{{\alpha^1}} \in &\ C_{per}(\mathcal Q_n) \text{ for }
|\alpha^1| = m_1,\\
p_{{\alpha^1}} \in &\
L^{r_\eta}(\mathcal Q_n) \text{ for }
|{\alpha^1}| = \eta < m_1,\ r_\eta \geq p,\ \frac{m_1 - \eta}{n-k} >
\frac{1}{ r_\eta}.
\end{aligned}
\right.
\end{equation}
Here $C_{per}(\mathcal Q_n) := \{ f \in C ([0,2\pi]^n) :\ f|_{x_j = 0 } = f|_{x_j = 2\pi}\  (j = 1,\ldots,n)\}$.
However, in order to apply perturbation results similar to
\cite{Denk-2003} or \cite{Nau-Saal-2011}, we assume $Q \equiv 1$
and set
\[
	\mathcal A(x,D) := P(x^1,D_1) + A_V(x^2,D_2).
\]
%%%%%%%%%%%%%%%%%%%%%%%%%%%%%%%%%%%%%%%%%%%%%%%%%%%%%%%%%%%%%%%%%%%%%%%%%%%%%%%
\begin{theorem}\label{mainresult_2}
Let $1 < p < \infty$, let $F$ be a UMD space
 enjoying property ($\alpha$),
let $\Omega := \mathcal Q_n \times V$, and
let the boundary valued problem $(A_V,B)$
fulfill the conditions of \cite[Theorem 8.2]{Denk-2003}
with angle of parameter-ellipticity $\varphi_{A_V}$.

For $P$ assume that
\begin{itemize}
	\item the coefficients satisfy \eqref{asscoeff1} and
	\item $P$ is parameter-elliptic
				in $ \mathcal Q_n$ with angle $\varphi_P \in [0,\pi-\varphi_{A_V})$.
\end{itemize}
Then for each $\varphi_0 >  \max\{\varphi_P,\varphi_{A_V}\}$
there exists $\delta = \delta(\varphi_0) \geq 0$ such that
the $L^p$-realization $\mathbb A +\delta$ of $\mathcal A + \delta$ is $\mathcal R$-sectorial with $\mathcal R$-angle
 $\phi_{\mathbb A + \delta}^{\mathcal{R} } \leq \varphi_0$. Moreover,
 we have
\begin{equation}\label{RZusatz}
\mathcal{R} ( \{\lambda^{1-  \frac{|\alpha|}{m_1}}D^{\alpha} (\lambda + \mathbb A +
\delta)^{-1}:\ \lambda \in \Sigma_{\pi - \phi},
\ \alpha\in\mathbb N_0^{n + n_v},\ 0 \leq
|\alpha| \leq m_1\} ) < \infty.
\end{equation}
In particular, if $\varphi_0 < \frac\pi 2$ then $\mathbb A +\delta$ has maximal $L^q$-regularity for every $1<q<\infty$.
\end{theorem}
%%%%%%%%%%%%%%%%%%%%%%%%%%%%%%%%%%%%%%%%%%%%%%%%%%%%%%%%%%%%%%%%%%%%%%%%%%%%%%%
\begin{proof}
In a first step, we consider $P(x,D)$ to be a homogeneous differential operator
with slightly varying coefficients. That is,
we consider
\[
	\mathcal A^{va}(x,D) := P_0(D_1) + R(x^1,D_1) + A_V(x^2,D_2),
\]
where $P_0(D_1) := \sum\limits_{|\alpha^1| = 2m}  p_{\alpha^1} D_1^{\alpha^1}$
is assumed to have constant coefficients and
$R(x^1,D_1) := \sum\limits_{|\alpha^1| = 2m}  r_{\alpha^1}(x^1) D_1^{\alpha^1}$
fulfills
$\sum\limits_{|\alpha^1| = 2m} \| r_{\alpha^1}\|_\infty \leq \eta$
with $\eta > 0$ sufficiently small. Then the claim follows due to
perturbation results for $\mathcal R$-sectorial operators
(see \cite{Denk-2003}, \cite{Nau-Saal-2011})
from Theorem \ref{mainresult_1} and Remark \ref{Qconst} b).

In a second step, we choose a finite but sufficiently fine open covering of ${\mathcal Q_n}$.
In view of the periodicity of the top order coefficients, we may assume every open set of the covering,
which intersects with $\mathbb R^n \setminus \mathcal Q_n$ to be cut at the boundary of $\mathcal Q_n$ and
continued within $\mathcal Q_n$ on the opposite side.
By means of reflection, this enables us to define local operators with slightly varying coefficients.
With the help of a partition of the unity and perturbation results for lower order terms
subject to condition \eqref{asscoeff1}, just as in \cite{Nau-Saal-2011}, the claim follows.
\end{proof}

\bibliographystyle{amsalpha}
%\bibliography{lit_diss_KW+mult+FR_20110308}

\begin{thebibliography}{FLM{\etalchar{+}}08}

\bibitem[AB02]{Arendt-Bu-2002}
W.~Arendt and S.~Bu, \emph{The operator-valued {M}arcinkiewicz multiplier
  theorem and maximal regularity}, Math. Z. \textbf{240} (2002), no.~2,
  311--343.

\bibitem[ABHN01]{Arendt-Batty-Hieber-Neubrander-2001}
W.~Arendt, C.~J.~K. Batty, M.~Hieber, and F.~Neubrander, \emph{Vector-valued
  {L}aplace transforms and {C}auchy problems}, Monographs in Mathematics,
  vol.~96, Birkh\"auser Verlag, Basel, 2001.

\bibitem[AR09]{Arendt-Rabier-2009}
W.~Arendt and P.~J. Rabier, \emph{Linear evolution operators on spaces of
  periodic functions}, Commun. Pure Appl. Anal. \textbf{8} (2009), no.~1,
  5--36.

\bibitem[BCS08]{Batty-Chill-Sri-2008}
C.~J.~K. Batty, R.~Chill, and S.~Srivastava, \emph{Maximal regularity for
  second order non-autonomous {C}auchy problems}, Studia Math. \textbf{189}
  (2008), no.~3, 205--223.

\bibitem[BK04]{Bu-Kim-2004}
S.~Bu and J.-M. Kim, \emph{Operator-valued {F}ourier multiplier theorems on
  {$L_p$}-spaces on {$\mathbb T^d$}}, Arch. Math. (Basel) \textbf{82} (2004),
  no.~5, 404--414.

\bibitem[Bu06]{Bu-2006}
S.~Bu, \emph{On operator-valued {F}ourier multipliers}, Sci. China Ser. A
  \textbf{49} (2006), no.~4, 574--576.

\bibitem[CS05]{Chill-Sri-2005}
R.~Chill and S.~Srivastava, \emph{{$L^p$}-maximal regularity for second order
  {C}auchy problems}, Math. Z. \textbf{251} (2005), no.~4, 751--781.

\bibitem[CS08]{Chill-Sri-2008}
\bysame, \emph{{$L^p$} maximal regularity for second order {C}auchy problems is
  independent of {$p$}}, Boll. Unione Mat. Ital. (9) \textbf{1} (2008), no.~1,
  147--157.

\bibitem[DHP03]{Denk-2003}
R.~Denk, M.~Hieber, and J.~Pr{\"u}ss, \emph{{$\mathcal R$}-boundedness,
  {F}ourier multipliers and problems of elliptic and parabolic type}, Mem.
  Amer. Math. Soc. \textbf{166} (2003), no.~788, viii+114.

\bibitem[FLM{\etalchar{+}}08]{Favin-Labbas-Maingot-Tanabe-Yagi-2008}
A.~Favini, R.~Labbas, S.~Maingot, H.~Tanabe, and A.~Yagi,
	\emph{A simplified approach in the study of elliptic differential
  equations in {UMD} spaces and new applications}, Funkcial. Ekvac. \textbf{51}
  (2008), no.~2, 165--187.

\bibitem[FSY09]{Favini-Shakmurov-Yakubov-2009}
A.~Favini, V.~Shakhmurov, and Y.~Yakubov, \emph{Regular boundary value
  problems for complete second order elliptic differential-operator equations
  in {UMD} {B}anach spaces}, Semigroup Forum \textbf{79} (2009), no.~1, 22--54.

\bibitem[FY10]{Favini-Yakubov-2010}
A.~Favini and Y.~Yakubov, \emph{Irregular boundary value problems for
  second order elliptic differential-operator equations in {UMD} {B}anach
  spaces}, Math. Ann. \textbf{348} (2010), no.~3, 601--632.

\bibitem[Gau01]{Gauss-2001}
T.~Gauss, \emph{Floquet theory for a class of periodic evolution equations in
  an {$L^p$}-setting}, Dissertation, KIT Scientific Publishing, Karlsruhe,
  2001.

\bibitem[Gui04]{Guidotti-2004}
P.~Guidotti, \emph{Elliptic and parabolic problems in unbounded domains}, Math.
  Nachr. \textbf{272} (2004), 32--45.

\bibitem[Gui05]{Guidotti-2005}
\bysame, \emph{Semiclassical fundamental solutions}, Abstr. Appl. Anal.
  \textbf{1} (2005), 45--57.

\bibitem[KL04]{Keyantuo-Lizama-2004}
V.~Keyantuo and C.~Lizama, \emph{Fourier multipliers and integro-differential
  equations in {B}anach spaces}, J. London Math. Soc. (2) \textbf{69} (2004),
  no.~3, 737--750.

\bibitem[KL06]{Keyantuo-Lizama-2006}
\bysame, \emph{Periodic solutions of second order differential equations in
  {B}anach spaces}, Math. Z. \textbf{253} (2006), no.~3, 489--514.

\bibitem[KLP09]{Keyantuo-Lizama-Poblete-2009}
V.~Keyantuo, C.~Lizama, and V.~Poblete, \emph{Periodic solutions of
  integro-differential equations in vector-valued function spaces}, J.
  Differential Equations \textbf{246} (2009), no.~3, 1007--1037.

\bibitem[KW04]{Kunstmann-Weis-2004}
P.~C. Kunstmann and L.~Weis, \emph{Maximal {$L_p$}-regularity for parabolic
  equations, {F}ourier multiplier theorems and {$H^\infty$}-functional
  calculus}, Functional Analytic Methods for Evolution Equations, Lecture Notes
  in Mathematics, vol. 1855, Springer, Berlin, 2004, pp.~65--311.

\bibitem[NS]{Nau-Saal-2011}
T.~Nau and J.~Saal, \emph{{$\mathcal R$}-sectoriality of cylindrical boundary
  value problems}, {P}arabolic {P}roblems. The Herbert Amann Festschrift to the
  occasion of his 70th birthday, Trends in Mathematics, Birkh\"auser, Basel, to
  appear.

\bibitem[{\v{S}}W07]{Strkalj-Weis-2007}
{\v{Z}}.~{\v{S}}trkalj and L.~Weis, \emph{On operator-valued {F}ourier
  multiplier theorems}, Trans. Amer. Math. Soc. \textbf{359} (2007), no.~8,
  3529--3547 (electronic).

\bibitem[Wei01]{Weis-2001}
L.~Weis, \emph{Operator-valued {F}ourier multiplier theorems and maximal
  {$L_p$}-regularity}, Math. Ann. \textbf{319} (2001), no.~4, 735--758.

\bibitem[XL98]{Xiao-Liang-1998}
T.-J. Xiao and J.~Liang, \emph{The {C}auchy problem for higher-order abstract
  differential equations}, Lecture Notes in Mathematics, vol. 1701, Springer,
  Berlin, 1998.

\bibitem[XL04]{Xiao-Liang-2004}
\bysame, \emph{Second order parabolic equations in {B}anach spaces with dynamic
  boundary conditions}, Trans. Amer. Math. Soc. \textbf{356} (2004), no.~12,
  4787--4809 (electronic).

\end{thebibliography}

\newcommand{\etalchar}[1]{$^{#1}$}
\providecommand{\bysame}{\leavevmode\hbox to3em{\hrulefill}\thinspace}
\providecommand{\MR}{\relax\ifhmode\unskip\space\fi MR }
\providecommand{\MRhref}[2]{%
  \href{http://www.ams.org/mathscinet-getitem?mr=#1}{#2}
}
\providecommand{\href}[2]{#2}

\end{document}